\documentclass[12pt]
{article}
 \usepackage{bbm}
 \usepackage{amssymb}
\usepackage{amssymb, amsthm, amsmath, amscd}
 \usepackage[usenames,dvipsnames]{color}
\newtheorem{thm}{Theorem}[section]
\newtheorem{lem}[thm]{Lemma}

\newtheorem{cor}[thm]{Corollary}
\theoremstyle{definition}

\theoremstyle{definition}
\newtheorem{df}[thm]{Definition}
\theoremstyle{definition}
\newtheorem{rem}[thm]{Remark}

\theoremstyle{definition}

\newcommand{\red}{\textcolor{red}}
\newcommand{\blue}{\color{blue}}

\renewcommand{\phi}{\varphi}

\newcommand{\N}{\mathbb{N}}

\newcommand{\Q}{\mathbb{Q}}

\newcommand{\C}{\mathbb{C}}
\newcommand{\T}{\mathbb{T}}

\newcommand{\Aff}{\operatorname{Aff}}

\newcommand{\id}{\operatorname{id}}

\newcommand{\hm}{homomorphism}
\newcommand{\dt}{\delta}
\newcommand{\ep}{\epsilon}

\newcommand{\norm}[1]{\left\Vert#1\right\Vert}

\newcommand{\Ratn}{\mathbb Q}
\newcommand{\eps}{\epsilon}

\newcommand{\Kzero}{\mathrm{K}_0}
\newcommand{\Kone}{\mathrm{K}_1}

\newcommand{\aff}{\mathrm{Aff}}

\newcommand{\andeqn}{\,\,\,{\rm and}\,\,\,}
\newcommand{\rforal}{\,\,\,{\rm for\,\,\,all}\,\,\,}
\newcommand{\CA}{$C^*$-algebra}

\newcommand{\af}{{\alpha}}

\newcommand{\wilog}{without loss of generality}

\newcommand{\D}{\mathbb D}
\newcommand{\beq}{\begin{eqnarray}}
\newcommand{\eneq}{\end{eqnarray}}

\newcommand{\p}{\mathfrak{p}}
\newcommand{\q}{\mathfrak{q}}

\usepackage{amsfonts}
\usepackage{mathrsfs}
\usepackage{textcomp}
\usepackage[all]{xy}

\numberwithin{equation}{section}

\begin{document}

\title{On the classification of  simple amenable C*-algebras with finite
decomposition rank, II
}
\author{George A. Elliott, Guihua Gong,  Huaxin Lin, and Zhuang Niu
 }
\date{}

\maketitle

\begin{abstract}
We prove that every unital simple separable  C*-algebra $A$ with finite decomposition rank which satisfies the UCT
has the property
that $A\otimes Q$ has generalized tracial rank at most one, where $Q$ is the universal UHF-algebra. Consequently,
$A$ is classifiable in the sense of Elliott.


\end{abstract}

\maketitle

\section{Introduction}

In a recent development in the Elliott program, the program of classification of amenable C*-algebras,
a certain class of finite unital simple separable amenable
C*-algebras, denoted by ${\mathcal N}_1$,  was shown to be classified by the Elliott invariant (\cite{GLNI} and \cite{GLNII}).
One important feature of this class of C*-algebras is that it exhausts all possible values of the Elliott invariant {{for}}
unital simple separable   C*-algebras which have finite decomposition rank (a property introduced in \cite{KW}; see Definition \ref{DefDr} below).


The purpose of this note is to show that, in fact, every unital simple separable (non-elementary) C*-algebra which has finite decomposition rank and satisfies the Universal Coefficient Theorem (UCT) is in the
class ${\mathcal N}_1.$ Since  every C*-algebra in ${\cal N}_1$ was shown in \cite{GLNI}   and \cite{GLNII} to be isomorphic to the inductive limit of a sequence of subhomogeneous {{C*}}-algebras with no dimension growth, the C*-algebras in ${\cal N}_1$ have
finite decomposition rank {{(see Remark \ref{Refiniterank} below).}}
In other words, the
class ${\mathcal N}_1$ is precisely the class of all unital simple separable (non-elementary) C*-algebras which have finite decomposition rank and satisfy the UCT, and hence we obtain a classification for all of these C*-algebras:

\begin{thm}\label{mainthm-dr}
Let $A$ be a unital simple separable  (non-elementary) C*-algebra with finite decomposition rank, and assume that $A$
satisfies  the UCT. Then $A\in\mathcal N_1.$ (See Definition \ref{Dnotation-1}, below.)
Hence (by {{Theorem 29.8}} of \cite{GLNII}), if $A$ and $ B$ are two (non-elementary) unital  simple separable  C*-algebras with
finite decomposition rank which satisfy the UCT, then $A\cong B$ if, and only if,
$$
(\Kzero(A), \Kzero(A)_+, [1_A]_0, \Kone(A), {\rm T}(A), r_A)\cong
(\Kzero(B), \Kzero(B)_+, [1_B]_0, \Kone(B), {\rm T}(B), r_B).
$$
\end{thm}

In fact,  we shall obtain (see Theorem \ref{TT}, below) the formally stronger result that every finite unital  simple separable (non-elementary) C*-algebra with finite nuclear dimension,
{{that}}
satisfies the UCT, and whose tracial states are all quasidiagonal, is in the class ${\cal N}_1.$
This result, combined with the recent
result of \cite{TWW} that the quasidiagonality hypothesis is redundant, yields that  the class ${\mathcal N}_1$ includes all finite unital simple separable   C*-algebras with finite nuclear dimension
which satisfy the UCT---see {{Theorem}} \ref{NT} below.
(The case of infinite unital simple separable C*-algebras with finite nuclear dimension was dealt with {over}
twenty years ago by  Kirchberg and Phillips---see  {{Remark}} \ref{Lrem} below.)

In a recent paper,  \cite{EN-K0-Z}, {{two of us ({{G.A.E. and Z.N.}})}}  proved
that every unital simple separable (non-elementary) C*-algebra $A$  with finite decomposition rank,
satisfying the UCT, and such that $\Kzero(A)$ has torsion free rank one,
belongs to ${\mathcal N}_1.$
The present paper is a  continuation of \cite{EN-K0-Z} with, now, a definitive  result.

It is perhaps worth mentioning that, the mathematical content of this paper  (for example, Theorem \ref {T2}) is independent of
that of \cite{GLNI} and \cite{GLNII}. Please also see Remark \ref{Rrtr=1}.

\vspace{0.2in}

{\bf Acknowledgements}.  During this research, G.A.E.~\hspace{-0.018in}received support from the Natural
Sciences and Engineering Research Council of Canada (NSERC), and G.G.~\hspace{-0.01in}and H.L.~\hspace{-0.017in}received support from the {{National Science Foundation (NSF)}}.  Z.N.~\hspace{-0.01in}received support from the Simons Foundation (Grant No.~317222) and the NSF (Grant No.~DMS-1800882). G.G., H.L., and Z.N.~\hspace{-0.01in}also received support from The Research Center for Operator Algebras at East China Normal University
 which is funded by {the National Natural Science Foundation of China (NNSF)} (Grant No.~11531003) and by the Science and Technology Commission of Shanghai
 Municipality (Grant No.~13dz2260400). {{G.G.~\hspace{-0.01in} is also supported by NNSF (Grant No.~11771117).}}
\section{Preliminaries}

\begin{df}\label{Dnotation}
As usual,
let $\Q$ denote the field of rational numbers.
Let us use the notation $Q$ for the UHF-algebra with $\Kzero(Q)=\Q$ and $[1_Q]=1.$


\end{df}

\begin{df}[N. Brown \cite{Bn}]\label{Dqdt} Let $A$ be a unital C*-algebra.
Denote by $\mathrm{T}(A)$
the tracial state space of $A,$  and denote by ${\rm T}_{{\rm qd}}(A)$ {{the subset of the quasidiagonal tracial states-- those
$\tau\in {\rm T}(A)$}} with the following property:
For any finite subset ${\mathcal F}$ and $\ep>0,$ there exists a unital completely positive map $\phi: A\to Q$ such that
\begin{equation*}
|\tau(a)-{\rm tr}(\phi(a))|<\ep,\quad a\in {\mathcal F},
\end{equation*}
and
\begin{equation*}
\|\phi(a)\phi(b)-\phi(ab)\|<\ep,\quad a,\, b\in {\mathcal F},
\end{equation*}
where ${ {\rm tr}}$ is the unique tracial state of $Q.$
\end{df}

\begin{df}\label{DET}
Let $F_1$ and $F_2$ be two finite dimensional C*-algebras and let $\psi_0, \psi_1: F_1\to F_2$ be two unital \hm s.
Consider the corresponding mapping torus,
$$
C={\rm C}(F_1, F_2, \psi_0, \psi_1)=\{(f,a)\in {\rm C}([0,1], F_2)\oplus F_1: f(0)=\psi_0(a)\andeqn f(1)=\psi_1(a)\}.
$$
Denote by  ${\mathcal C}$  the class of unital C*-algebras obtained in this way.
C*-algebras in the class ${\mathcal C}$ are often called Elliott-Thomsen building blocks. They are also called
one-dimensional non-commutative CW complexes.

Denote by ${\mathcal C}_0$ the subclass of ${\mathcal C}$ consisting of those C*-algebras $C\in {\mathcal C}$ such
that $\Kone(C)=\{0\}.$

We shall in fact only work with the $Q$-stabilizations of these algebras, which can be described just by replacing $F_1$ and $F_2$ with finite direct sums of copies of $Q$.
\end{df}

\begin{df}[9.1 of \cite{GLNI}]\label{DgTR}
Let $A$ be a (non-elementary) unital simple C*-algebra. We shall say that $A$ has generalized tracial rank at most one,
if the following property holds:

Let $\ep>0,$ let $a\in A_+\setminus \{0\}$ and let ${\mathcal F}\subseteq A$ be a
finite subset.
There exist a non-zero projection $p\in A$ and a sub-C*-{{algebra}} $C\in {\mathcal C}$
with $1_C=p$ such that
\begin{eqnarray*}
&&\|xp-px\|<\ep,\quad x\in {\mathcal F}, \\
&&{\rm dist}(pxp, C)<\ep,\quad x\in {\mathcal F}, \andeqn \\
&&1-p\lesssim a.
\end{eqnarray*}
The last condition means that  there exists a partial isometry $v\in A$ such that
$v^*v=1-p$ and $vv^*\in \overline{aAa}.$
If $A$ has generalized tracial rank at most one, we will write ${\mathrm{gTR}}(A)\le 1.$
It was shown in \cite{GLNI} that if ${\mathrm{gTR}}(A)\le 1$, then $A$ is quasidiagonal, and is ${\mathcal Z}$-stable if it is also amenable.
\end{df}

\begin{df}\label{dpk}
Let $A$ and $B$ be unital C*-algebras and let $L: A\to B$ be a contractive completely positive map.
Let ${\mathcal G}$ be a finite subset of $A$ and $\dt>0.$ Recall that
$L$ is said to be ${\mathcal G}$-$\dt$-multiplicative if $\|L(x)L(y)-L(xy)\|<\dt$ for all $x, y\in {\mathcal G}.$
Given a finite subset ${\mathcal P}$ of projections in $A,$ if ${\mathcal G}$ is sufficiently large and $\dt$ is sufficiently small, then
there is a projection $q\in B$ such that $\|L(p)-q\|<1/4.$ Moreover, for each projection {{$p \in {\mathcal G}$}}, if $\delta<1/4$, then the projection $q$ can be chosen such that
\begin{equation}\label{pert-proj}
\|L(p) - q\|<2\delta.
\end{equation}
Note that if $q'\in B$ is another projection such that
$\|L(p)-q'\|<1/4,$ then {{$q'$ and $q$}} are unitarily equivalent. Recall that  $[L(p)]$ often denotes this equivalence
class of projections (see e.g. \cite{LinTAF2}). As usual, when $[L(p)]$ is written, it is understood that ${\mathcal G}$ is sufficiently large and
$\dt$ is sufficiently small that $[L(p)]$ is well defined.

\end{df}

\begin{df}[\cite{GLNI}]\label{Dnotation-1}

Let $A$ be a unital simple separable C*-algebra. Let us say that $A$ has rational generalized tracial rank at most one if ${\mathrm{gTR}}(A\otimes Q)\le 1.$

Let us say that $A$ belongs to the class ${\cal N}_1$  if, in addition, it {{is amenable and}}
satisfies the UCT (\cite{RS}), and is Jiang-Su stable, i.e., {{is}} invariant under
tensoring with the Jiang-Su C*-algebra (\cite{JS}; see also \cite{Ell}).
  As pointed out above, it follows from \cite{GLNII} (together with \cite{Wdecomp}) that, instead of the last property, it is equivalent to
assume finite decomposition rank (or, by \cite{W}, even just finite nuclear
 dimension); see Definition \ref{DefDr} below.
(By now, we know  (see \cite{MS2}, \cite{SWW} and \cite{CETWW}) that
 a unital separable simple nuclear \CA\,  is Jiang-Su stable if and only if it has
 finite nuclear dimension---see
 {{red} ``Added November 2, 2021"} at the end of this paper.)


\end{df}

The following are the main results of \cite{GLNI} and \cite{GLNII}:

\begin{thm}\label{Tiso}
Let $A$ and $B$ be two unital C*-algebras in ${\mathcal N}_1.$ Then
$A\cong B$ if and only if  ${\rm Ell}(A)\cong {\rm Ell}(B),$ i.e., $A\cong B$ if and only if
$$
(\Kzero(A), \Kzero(A)_+, [1_A]_0, \Kone(A), {\rm T}(A), r_A)\cong
(\Kzero(B), \Kzero(B)_+, [1_B]_0, \Kone(B), {\rm T}(B), r_B).
$$
Moreover, any isomorphism between $\mathrm{Ell}(A)$ and $\mathrm{Ell}(B)$ can be lifted to an isomorphism between $A$ and $B$.
\end{thm}

\begin{proof}
The  first part of the statement follows from {{Theorem 29.8 of \cite{GLNII}.}}

The second part of the statement needs some explanation.
Note that $A$ and $B$ satisfy  the assumption of Theorem 29.5 of \cite{GLNII}, {{by Corollary 19.3 of \cite{GLNI}.}} 
 Let $\Gamma: {\rm Ell}(A) \to {\rm Ell}(B)$ be an isomorphism. Repeat the proof of Theorem 29.5 of \cite{GLNII}  until the second-last sentence
of that proof, namely:
``One then obtains a unitary suspended isomorphism which lifts  $\Gamma$ along ${\cal Z}_{\p,\q}$ (see \cite{Wlocal})."
For the present purpose (note that the lifting statement is not explicitly formulated in \cite{GLNII}), replace the last sentence of that proof by the sentence ``It follows from Theorem 7.1 of \cite{Wlocal} that $A\otimes {\cal Z}$  and $B\otimes {\cal Z}$ are isomorphic and
the isomorphism {{lifts}}  $\Gamma$."
\end{proof}

\begin{thm}[Theorem 13.50 of \cite{GLNI}]\label{Trange}
For any non-zero countable weakly unperforated simple ordered group
$G_0$ with order unit $u,$ any countable abelian group $G_1,$ any non-empty
metrizable Choquet simplex $T$, and any surjective affine map
$r: T\to \mathrm{S}_u(G_0)$ ($\mathrm{S}_u(G_0)$ is the state space of $G_0$---always non-empty), there
exists a (unique) unital simple C*-algebra  $C$ in ${\mathcal N}_1$, which is the inductive limit of
a sequence of subhomogeneous C*-algebras with two-dimensional spectrum,
such that
$$
{\rm Ell}(C)=(G_0, (G_0)_+, u, G_1, T, r).
$$
\end{thm}
\begin{df}
Let $A$ and $B$ be C*-algebras. Recall (\cite{KW}) that a completely positive map $\phi: A\to B$ is said to have order zero if
$$ab=0\  \Rightarrow\  \phi(a)\phi(b)=0,\quad a, b\in A.$$
\end{df}

\begin{df}[\cite{KW}, \cite{WZ}]\label{DefDr}
A C*-algebra $A$ has nuclear dimension at most $n$, if there exists a net $(F_\lambda , \psi_\lambda , \phi_\lambda)$, $\lambda\in\Lambda$, such that the $F_\lambda$ are finite dimensional C*-algebras, and such that $\psi_\lambda: A\to F_\lambda$ and $\phi_\lambda: F_\lambda \to A$ are completely positive maps satisfying \\
$~\ \ $(1) $\phi_\lambda\circ \psi_\lambda \to \id_A$ pointwise (in norm),\\
$~\ \ $(2) $\|\psi_\lambda\|\leq 1$, and\\
$~\ \ $(3) for each $\lambda$, there is a decomposition
$F_\lambda=F_\lambda^{(0)}\oplus\cdots\oplus F_\lambda^{(n)}$
such that each restriction
$\phi_\lambda|_{F_\lambda^{(j)}}$ is a contractive order zero map.

Moreover, if the the map $\phi_\lambda$ can be chosen to be contractive itself, then $A$ is said to have  decomposition rank at most $n$.

Recall that finite nuclear dimension immediately implies nuclearity, which by \cite{Connes-nuc} and \cite{Haag-nuc} is equivalent to amenability.  The  nuclear dimension of a certain $C^*$-algebra associated with a discrete metric space is related to  asymptotical dimension  of the underline space, and the concept of asymptotical dimension  has fundamental applications to geometry and topology (see \cite{Yu1} and \cite{Yu2}).
\end{df}

The main theorem of this paper is that the class  ${\mathcal N}_1$ of C*-algebras actually contains (and hence coincides with) the class of all (non-elementary)  {{unital}} simple separable  C*-algebras with finite decomposition rank which also satisfy the UCT. 
In particular it follows (on using both \ref{Tiso} and \ref{Trange}) that every such C*-algebra is the inductive limit of a sequence of subhomogeneous C*-algebras (with no dimension growth). 

\section{Some existence theorems}

Denote by ${\cal K}$ the C*-algebra of all compact operators {{on}} $l^2,$ an infinite dimensional
separable Hilbert space.  Let $\{e_{i,j}\}$ be {{the canonical}} system of matrix units for ${\cal K}.$
We will use the fact that ${\cal K}\otimes {\cal K}\cong {\cal K}$
and assume that such an isomorphism has been fixed.
{{Let $B$ be a C*-algebra.  We may identify  $B\otimes {\cal K}$ with $(B\otimes {\cal K})\otimes e_{1,1}.$
Let $A$ be another C*-algebra and
$\psi_1, \psi_2, ...,\psi_n: A\to B\otimes {\cal K}$ be linear maps.
For convenience, when there is no confusion,  {{using the}} identification above, we {{shall write}} $\Phi:=\psi_1\oplus \psi_2\oplus \cdots \oplus \psi_n:A \to B\otimes {\cal K}$ {{to denote the}} orthogonal sum:
\beq
\Phi(a)=\sum_{i=1}^n \psi_i(a)\otimes e_{i,i}
\rforal a\in A,
\eneq
}}

{{Likewise, for projections $p_1,p_2, ...,p_m\in B\otimes {\cal K},$ we {{shall}} write
$P:=p_1\oplus p_2\oplus \cdots \oplus p_m\in B\otimes {\cal K}$ 
if  there is no confusion.
Therefore, in the case that $\psi_i(A)\subseteq p_i(B\otimes {\cal K})p_i,$ $1\le i\le n,$
we may view  $\Phi$ as mapping $A$ to $P(B\otimes {\cal K})P.$}}

{{We will use this convention repeatedly.}}

\begin{lem}\label{0.5}
Let $A$ be a unital simple  separable amenable quasidiagonal C*-algebra satisfying the UCT. Assume that $A\cong A\otimes Q$.

Let a finite subset $\mathcal G$ of $A$ and $\ep_1, \ep_2>0$ be given. Let $p_1, p_2, ..., p_s\in  A$
be projections such that
$[1], [p_1], [p_2], ..., [p_s]\in \Kzero(A)$ are $\Q$-linearly independent. (Recall that $\Kzero(A)\cong\Kzero(A\otimes Q)\cong\Kzero(A)\otimes \Ratn$.) There are a $\mathcal G$-$\ep_1$-multiplicative completely positive map $\sigma: A\to Q$ with $\sigma(1)$ a projection satisfying
$$\mathrm{tr}(\sigma(1))<\ep_2$$
(where $\mathrm{tr}$ denotes the unique {{tracial}} state on $Q$), and $\delta>0$,
such that, for any $r_1, r_2, ..., r_s\in\Q$ with
$$|{r_i}|<\dt,\  i=1, 2,..., s,$$
there is a $\mathcal G$-$\ep_1$-multiplicative completely positive map $\mu: A\to Q $, with $\mu(1)=\sigma(1)$, such that
$$[\sigma(p_i)]-[\mu(p_i)]=r_i,\quad i=1,2, ...,s.$$

\end{lem}

\begin{proof}
Let us agree that $\sigma$ and $\mu$ (to be constructed below) are also understood to be required to be sufficiently multiplicative on $p_1, p_2, ..., p_k$ that the classes $[\sigma(p_i)]$ and $[\mu(p_i)]$ make sense (see \ref{dpk} above). (Similarly for other completely positive  approximately multiplicative maps, to be introduced below.)

Denote by $G_0$  the subgroup of $\Kzero(A)$ generated by $\{[1_A], [p_1],...,[p_s]\}.$
%
Since $[1_A], [p_1], ...,[p_s]$ are $\Q$-linearly independent, for each $i=1,2,...,s,$ there exists a \hm\,
$\af_i: \Kzero(A)\to \Q\cong \Kzero(Q)$ such that
\beq\label{05-2}
\af_i([p_i])=1,\quad \af_i([1_A])=0,\quad\mathrm{and}\quad \af_i([p_j])=0,\,\,\,j\not=i.
\eneq
We may regard $\af_i$ as an element of $\mathrm{KL}(A,Q)$ (see \cite{DL}).
Since $A$ is a unital simple amenable quasidiagonal C*-algebra, by \cite{BK2},  $A$ is a unital simple strong NF-algebra.
It follows from \cite{BK2} that $A=\overline{\bigcup_{n=1}^{\infty} A_n},$ where $\{A_n\}$  is an increasing sequence of
unital, amenable, residually finite-dimensional  C*-algebras.
It follows from Theorem 5.9 of \cite{LinTAF2} that there are ${\mathcal G}$-$\ep_1$-multiplicative completely positive maps
$\sigma_i,\mu_i: A\to Q\otimes {\mathcal K}$ such that
$\sigma_i(1_A)$ and $\mu_i(1_A)$ are projections, and
\beq\label{05-3}
[\sigma_i]|_{G_0}-[\mu_i]|_{G_0}=\af_i|_{G_0},\quad i=1,2,...,s.
\eneq
Since $\af_i([1_A])=0,$ we have $[\sigma_i(1_A)]=[\mu_i(1_A)].$ Therefore, \wilog, we may assume
that
$$\sigma_i(1_A)=\mu_i(1_A)=:P_i,\quad i=1,2,...,s.$$

{{Consider}} the projection
$$
P:= \bigoplus_{i=1}^s (P_i\oplus P_i),
$$
and the unital $\mathcal G$-$\ep_1$-multiplicative completely positive  map
$$
\bigoplus_{i=1}^s(\sigma_i\oplus\mu_i): A\to P(Q\otimes \mathcal K) P,
$$
 where $\sigma_i \oplus \mu_i$ means the map $a\mapsto \sigma_i(a)\oplus\mu_i(a)$
 {{(see the beginning of this section).}}
Note that $P(Q\otimes \mathcal K) P\cong Q$. Choose a projection $R\in Q\otimes\mathcal K$ with $0 < \mathrm{tr}(R)
{{\le}}  \min\{1, \ep_2\}$ and a rescaling  homomorphism
$$S: Q\otimes \mathcal K \to Q\otimes\mathcal K,\ P \mapsto R.$$
Consider the map
$$\sigma:=S\circ (\bigoplus_{i=1}^s(\sigma_i\oplus\mu_i)): A\to Q\otimes\mathcal K$$
and the strictly positive number
$$\delta := \frac{\mathrm{tr}(R)}{\mathrm{tr}(P)}.$$
(Here, $\mathrm{tr}$ denotes the tensor product of the traces on $Q$ and $\mathcal K$, normalized {{to be $1$ on $1_Q\otimes e_{11}$}}.)
Note that since $\mathrm{tr}(R) {{\le}} 1$, one has
$$\mathrm{tr}(\sigma(1))=\mathrm{tr}(R){{\le}}1,$$ and so we may regard $\sigma$ as a map from $A$ to $Q$ (rather than $Q\otimes\mathcal K$).

Let us show that $\sigma$ and $\delta$ satisfy the condition of the lemma.

Let $r_1, r_2, ..., r_s\in\Q$ be given with
$$|r_i|<\delta,\quad  i=1, 2, ..., s.$$

For each $i=1, 2, ..., s$, choose a projection $R_i\in Q\otimes\mathcal K$ with $\mathrm{tr}(R_i)=|r_i|$, and choose a rescaling homomorphism
$$S_i: Q\otimes\mathcal K\to Q\otimes\mathcal K,\ 1\otimes e \mapsto R_i,$$
where $e$ is a minimal non-zero projection of $\mathcal K$. For each $i=1, 2, ..., s$, consider the pair of maps
$$S_i\circ \sigma_i,\ S_i\circ\mu_i: A\to Q\otimes\mathcal K.$$
Then, for each $i=1, 2, ..., s$,
\beq
&&{[}S_i\circ \sigma_i(p_i){]}-{[}S_i\circ \mu_i(p_i){]}=|r_i|, \nonumber \\
&&[S_i\circ \sigma_i(1){]} - {[}S_i\circ \mu_i(1){]} = 0, \nonumber \\
&&[S_i\circ \sigma_i(p_j){]} - {[}S_i\circ \mu_i(p_j){]}=0, \quad j=1, 2, ..., s,\  j\neq i. \nonumber
\eneq


Consider the direct sum maps
$$\tilde{\sigma}:=(\bigoplus_{r_i>0} S_i\circ \sigma_i)\oplus (\bigoplus_{r_i<0} S_i\circ \mu_i)$$
and
$$\tilde{\mu}:=(\bigoplus_{r_i>0} S_i\circ \mu_i)\oplus (\bigoplus_{r_i<0} S_i\circ \sigma_i).$$
It follows from \eqref{05-2} and \eqref{05-3} that
$$[\tilde{\sigma}(p_i)] - [\tilde{\mu}(p_i)]=r_i,\quad i=1, 2, ..., s.$$

Note that
\begin{equation}\label{eq-decp-0}
\sigma = S\circ (\bigoplus_{i=1}^s(\sigma_i\oplus\mu_i)) = \bigoplus_{i=1}^s((S\circ\sigma_i)\oplus(S \circ \mu_i)).
\end{equation}
For each $i=1, 2, ..., s$, since $$\mathrm{tr}(S_i(P))=\mathrm{tr}(P)\cdot\mathrm{tr}(S_i(1\otimes e))=\mathrm{tr}(P)\cdot\mathrm{tr}(R_i)=\mathrm{tr}(P)|r_i| <  \mathrm{tr}(P)\delta = \mathrm{tr}(R)=\mathrm{tr}(S(P)),$$
there is a rescaling homomorphism $T_i: Q\otimes \mathcal K \to Q\otimes\mathcal K$ such that
$$[S]=[S_i] + [T_i] = [S_i\oplus T_i]\quad\mathrm{on}\  \Kzero(Q)\cong \Ratn.$$
Therefore, by \eqref{eq-decp-0}, on $G_0$,
\begin{eqnarray*}
[\sigma] & = & \sum_{i=1}^s([(S\circ\sigma_i)] +[(S \circ \mu_i))]) \\
 & = & \sum_{i=1}^s([S_i\circ\sigma_i] + [T_i\circ\sigma_i]) + \sum_{i=1}^s ([S_i\circ \mu_i ] + [T_i \circ \mu_i]) \\
 & = & [(\bigoplus_{r_i>0}((S_i\circ\sigma_i) \oplus (T_i\circ\sigma_i)))\oplus (\bigoplus_{r_i\leq 0}((S_i\circ\sigma_i) \oplus (T_i\circ\sigma_i))) \\  &    & \oplus (\bigoplus_{r_i<0} ((S_i\circ \mu_i )\oplus (T_i \circ \mu_i))) \oplus (\bigoplus_{r_i \geq 0} ((S_i\circ \mu_i )\oplus (T_i \circ \mu_i)))] \\
 & = & [\tilde{\sigma}] \oplus [\gamma],
\end{eqnarray*} 
where $$\gamma= (\bigoplus_{r_i>0}(T_i\circ\sigma_i))\oplus (\bigoplus_{r_i\leq 0}((S_i\circ\sigma_i) \oplus (T_i\circ\sigma_i))) \oplus (\bigoplus_{r_i<0}  (T_i \circ \mu_i)) \oplus (\bigoplus_{r_i \geq 0} ((S_i\circ \mu_i )\oplus (T_i \circ \mu_i))).$$


%

Consider the direct sum completely positive map
$$\mu:=\tilde{\mu}\oplus \gamma.$$
We have
$$[\sigma(p_i)] - [\mu(p_i)]=[\tilde{\sigma}(p_i)] - [\tilde{\mu}(p_i)]=r_i,\quad i=1, 2, ..., s,$$
as desired (with $\mu$ regarded as a map from $A$ to $Q$, as $\mu(1)=\sigma(1)$).
\end{proof}

\begin{rem}
The assumption that $A$ be amenable in Lemma \ref{0.5}  can be removed.
In the proof one can apply Theorem 5.5 of \cite{DE}
in place of Theorem 5.9 of \cite{LinTAF2}.
\end{rem}

 Let $l, r=1, 2,...$ be given. In the rest of the paper,
we identify $\Kzero(Q^l)$ with $\Ratn^l$ (and $\Kzero(Q^r)$ with $\Ratn^r$) by identifying $[1_{Q^l}]$ with $(\underbrace{1, 1, ..., 1}_l)$ and ($[1_{Q^r}]$ with $(\underbrace{1, 1, ..., 1}_r)$), where $Q^l=\underbrace{Q\oplus\cdots\oplus Q}_l$ and $Q^r=\underbrace{Q\oplus\cdots\oplus Q}_r$. If $\psi: Q^l \to Q^r$ are unital, then
$$(\psi)_{*0}(\underbrace{1, 1, ..., 1}_l)=(\underbrace{1,1, ..., 1}_r),$$ and therefore
\begin{equation}\label{mul-unit}
(\psi)_{*0}(\underbrace{t, t, ..., t}_l)=(\underbrace{t, t, ..., t}_r),\quad t\in \Q.
\end{equation}


\begin{lem}\label{0.6}
Let $A$ be a unital simple  separable  amenable quasidiagonal C*-algebra satisfying the UCT.
Assume that $A\cong A\otimes Q.$

Let ${\mathcal G}$ be a finite subset of $A$, let $\ep_1, \ep_2>0,$ and let ${{p_1, p_2,..., p_s}}\in A$
be projections such that $[1_A], [p_1], [p_2],...,[p_s]\in \Kzero(A)$ are $\Q$-linearly independent.
There exists  $\dt>0$ satisfying the following condition.

Let $\psi_k: Q^l \to Q^r,$ $k=0,1,$ be unital homomorphisms, where  $l, r=1, 2, ...$ .
Set
$$\D=\{x\in \Q^l: (\psi_0)_{*0}(x)=(\psi_1)_{*0}(x)\}\subseteq \Q^l.$$
{There exists
a}  ${\mathcal G}$-$\ep_1$-multiplicative completely positive map $\Sigma: A\to Q^l$, such that $\Sigma(1_A)$ is a projection, with the following properties:
$$\tau(\Sigma(1_A))<\ep_2,\quad \tau \in {\rm T}(Q^l),$$
$$ [\Sigma(1_A)],\ [\Sigma(p_j)]\in  \D, \quad j=1,2,...,s,$$
 and, for any $r_1, r_2,...,r_s\in \Q^r$ satisfying
$$
|r_{i,j}|<\dt,
$$
where $r_i=(r_{i,1},r_{i,2},...,r_{i,r}),$ $i=1,2,...,s,$
there is a ${\mathcal G}$-$\ep_1$-multiplicative completely positive map $\mu: A\to Q^r$, with $\mu(1_A)$ a projection, such that
$$ [\psi_0\circ \Sigma(p_i)]-[\mu(p_i)]
=r_i, \quad i=1,2,...,s,$$
and $$[\mu(1_A)]=[\psi_0\circ \Sigma(1_A)].
$$
\end{lem}

\begin{proof}
Put $p_0=1_A$ and ${\mathcal P}=\{[1_A], [p_1],...,[p_s]\}.$
Applying Lemma \ref{0.5}, we obtain a ${\mathcal G}$-$\ep_1$-multiplicative $\sigma: A\to Q$ and $\delta>0$
satisfying the conclusion of Lemma \ref{0.5} with respect to $\mathcal G$, $\eps_1$, $\eps_2$, and $\mathcal P$.

Let us show that $\delta$ is as desired.

For a given integer $l=1, 2, ...$, consider the map $\Sigma: A\to  Q^l$, the sum of $l$ copies of $\sigma$,
$$\Sigma=\sigma \oplus \sigma \oplus \dots \oplus \sigma.$$
Let us show that $\Sigma$ has the required properties.
Let $r=1, 2,$ ... and $\psi_k: Q^l\to Q^r$, $k=0, 1$, be given (as in the statement of the condition on $\delta$ to be verified).
Since $\psi_0$ and $\psi_1$ are assumed to be unital, $[1_{Q^r}]=(1, 1, ..., 1)\in \D.$
It then follows that
$$[\Sigma(p_i)]=([\sigma(p_i)], [\sigma(p_i)], ..., [\sigma(p_i)])=[\sigma(p_i)](1, 1, ..., 1) \in \D,\quad i=0, 1,2, ..., s,$$
where $[\sigma(p_i)]$ is regarded as a rational number.
In other words,
$$[\psi_0\circ \Sigma(p_i)]=[\psi_1\circ \Sigma(p_i)],\quad i=0, 1,2,...,s.$$

Since $\mathrm{tr}([\sigma(1_A)])<\eps_2$, one has
$$\tau([\Sigma(1_A)])<\eps_2,\quad  \tau\in {\rm T}(Q^l).$$


Let $r_1, r_2,...,r_s\in \Q^r$ be given such that $|r_{i,j}|<\dt,$ $j=1,2,...,r$ and $i=1,2,...,s.$ Let us show that $\mu$ exists as required.

Fix $j=1, 2, ..., r$ and let $\mu_j: A\to Q$ (in place of $\mu$) denote the ${\mathcal G}$-$\ep_1$-multiplicative completely positive map
given by Lemma \ref{0.5}
for the $s$-tuple $r_{1,j}, r_{2,j},...,r_{s,j}.$
That is, $\mu_j(1_A)=\sigma(1_A)$, {{and}}
\begin{equation}\label{cord-diff}
[\sigma(p_i)]-[\mu_j(p_i)]=r_{i,j} \in \Kzero(Q),\quad i=1,2,..., s.
\end{equation}
%
%
%
Define $\mu: A\to Q^r$  by
$$\mu(a)=(\underbrace{\mu_1(a), \mu_2(a),..., \mu_r(a)}_r), \quad a \in A.$$

Then, for each $i=1, 2, ..., s$,
\begin{eqnarray*}
[\psi_0\circ\Sigma(p_i)]-[\mu(p_i)] & = & (\psi_0)_{*0} (\underbrace{[\sigma(p_i)], [\sigma(p_i)], ..., [\sigma(p_i)]}_l) - (\underbrace{[\mu_1(p_i)], [\mu_2(p_i)],..., [\mu_r(p_i)]}_r) \\
& = & (\underbrace{[\sigma(p_i)], [\sigma(p_i)], ..., [\sigma(p_i)]}_r) - (\underbrace{[\mu_1(p_i)], [\mu_2(p_i)],..., [\mu_r(p_i)]}_r) \quad \textrm{(by \eqref{mul-unit})} \\
& = & (([\sigma(p_i)] - [\mu_1(p_i)]), ([\sigma(p_i)] - [\mu_2(p_i)]), ..., ([\sigma(p_i)]- [\mu_r(p_i)]))\\
& = & (r_{i, 1}, r_{i, 2}, ..., r_{i, r}) = r_i\quad\quad\textrm{(by \eqref{cord-diff})},
\end{eqnarray*}
as desired. A similar computation shows that $[\psi_0\circ\Sigma(1_A)]=[\mu(1_A)]$.
\end{proof}

\begin{lem}\label{0.7}

Let $A$ be a unital simple separable  amenable quasidiagonal C*-algebra satisfying the UCT.
Assume that $A\cong A\otimes Q.$

Let ${\mathcal G}\subseteq A$ be a finite subset, let $\ep_1, \ep_2>0$ and let $p_1, p_2,..., p_s \in A$
be
projections such that $[1_A], [p_1], [p_2],...,[p_s]\in \Kzero(A)$ are $\Q$-linearly independent.
Then there exists $\dt>0$ satisfying the following condition.

 Let $\psi_k: Q^l\to Q^r$, $k=0, 1$, be unital homomorphisms, where $l, r=1, 2, ...$.
Set $$\D=\{x\in \Q^l: (\psi_0)_{*0}(x)=(\psi_1)_{*0}(x)\}\subseteq \Q^l.$$
{There exists
a}
${\mathcal G}$-$\ep_1$-multiplicative completely positive map $\Sigma: A\to Q^l$, such that $\Sigma(1_A)$ is a projection, with the following properties:
$$\tau(\Sigma(1_A))<\ep_2,\quad \tau \in {\rm T}(Q^l),$$
$$[\Sigma(1_A)],\  [\Sigma(p_i)]\in  \D,\quad  i=1,2,...,s,$$
and, for any $r_1, r_2,...,r_s\in \Q^l$ satisfying
$$
|r_{i,j}|<\dt,
$$
where $r_i=(r_{i,1},r_{i,2},...,r_{i,l}),$ $i=1,2,...,s,$
there is a ${\mathcal G}$-$\ep_1$-multiplicative completely positive map $\mu: A\to  Q^l$, with $\mu(1_A)=\Sigma(1_A)$, such that
\beq\nonumber
[\Sigma(p_i)]-[\mu(p_i)]=r_i,\quad {{i}}=1,2,...,s.
\eneq

\end{lem}

\begin{proof}
This is similar to the proof of Lemma \ref{0.6}.
Put $[p_0]=[1_A]$ and ${\mathcal P}=\{[p_0], [p_1],...,[p_s]\}.$
Applying Lemma \ref{0.5}, we obtain a ${\mathcal G}$-$\ep_1$-multiplicative $\sigma: A\to Q$ and $\delta>0$
satisfying the conclusion of Lemma \ref{0.5} with respect to $(\mathcal G, \eps_1, \eps_2)$.

Let us show that $\delta$ is as desired.

Consider the map $\Sigma: A\to Q^l$, the sum of $l$ copies of $\sigma$,
$$\Sigma=\sigma\oplus\sigma\oplus\cdots\oplus\sigma,$$
for a given $l=1, 2, ...$ .
Then the same argument as that of {{Lemma \ref{0.6}}} 
shows that
$$[\Sigma(p_i)]\in \D,\quad i=0,1,2,...,s.$$
It is also clear that
$$\tau(\Sigma(1_A))<\eps_2,\quad \tau\in\mathrm{T}(Q^l).$$



Let $r_1, r_2,...,r_k\in \Q^l$ be given
such that $|r_{i,j}|<\dt,$ $j=1,2,...,l$ and $i=1,2,...,s.$ Let us show that $\mu$ exists as required.

Fix $j=1, 2, ..., r$ and let $\mu_j: A\to Q$ (in place of $\mu$) denote the ${\mathcal G}$-$\ep_1$-multiplicative completely positive map
given by Lemma \ref{0.5}
for the $s$-tuple $r_{1,j}, r_{2,j},...,r_{s,j}.$
That is, $\mu_j(1_A)=\sigma(1_A)$, and
\begin{equation}\label{cord-diff-1}
[\sigma(p_i)]-[\mu_j(p_i)]=r_{i,j} \in \Kzero(Q),\quad i=1,2,..., s.
\end{equation}
%
%
%
Define $\mu: A\to Q^l$  by
$$\mu(a)=(\underbrace{\mu_1(a), \mu_2(a),..., \mu_r(a)}_l), \quad a \in A.$$

Then, for each $i=1, 2, ..., s$,
\begin{eqnarray*}
[\Sigma(p_i)]-[\mu(p_i)] & = & (\underbrace{[\sigma(p_i)], [\sigma(p_i)], ..., [\sigma(p_i)]}_l) - (\underbrace{[\mu_1(p_i)], [\mu_2(p_i)],..., [\mu_r(p_i)]}_l) \\
& = & (([\sigma(p_i)] - [\mu_1(p_i)]), ([\sigma(p_i)] - [\mu_2(p_i)]), ..., ([\sigma(p_i)]- [\mu_r(p_i)]))\\
& = & (r_{i, 1}, r_{i, 2}, ..., r_{i, l}) = r_i\quad\quad\textrm{(by \eqref{cord-diff-1})},
\end{eqnarray*}
as desired. Moreover, since $\mu_j(1_A)=\sigma(1_A),$ we have
$\Sigma(1_A)=\mu(1_A).$
\end{proof}

\section{The main result}
Let us begin by recalling the (stable) uniqueness result used in \cite{EN-K0-Z}:
\begin{lem}[Corollary 2.6 of \cite{EN-K0-Z}; see also Lemma 4.14 of \cite{GLNI}, and Definition 5.6   and  Theorem 5.9 of \cite{Linuct}]\label{stable-uniq-Q}
Let $A$ be a unital simple separable amenable C*-algebra which {{satisfies}} the UCT. Assume that $A\cong A\otimes Q.$

 For any $\ep>0,$ any finite subset $\mathcal F$ of $A$,
 there exist $\dt>0,$
 a finite subset {{${\cal G}$ of $A$}}, a finite subset ${\cal P}$ of projections in $A$,
 and an integer $n\in \N$ 
 with the following property.

 For any three completely positive contractions $\phi, \psi, \xi: A\to Q$ which are
 ${\cal G}$-$\dt$-multiplicative,
with $\phi(1)=\psi(1)=1_Q-\xi(1)$ a projection, $[\phi(p)]_0=[\psi(p)]_0$ in $\Kzero(Q)$ for all $p\in\mathcal P$, and $\mathrm{tr}(\phi(1))=\mathrm{tr}(\psi(1))<1/n$, where $\mathrm{tr}$ is the unique tracial state of $Q$, there exists a unitary $u\in Q$ such that
$$
\norm{u^* (\phi(a)\oplus\xi(a))u-\psi(a)\oplus\xi(a)}<\eps,\quad a\in\mathcal F.
$$
\end{lem}

The following two existence results are related to Lemma 16.9 of \cite{GLNI} (and its proof).
%

We will use the following known facts:
$Q\otimes Q\cong Q,$ and
any unital endomorphism $\phi:Q\to Q$ is approximately unitarily equivalent to the identity map.
\begin{lem}\label{partuniq}
Let $A$ be a unital simple separable amenable C*-algebra
which satisfies the UCT. Assume that $A\otimes Q\cong A.$

For any $\ep>0$ and any finite subset ${\mathcal F}$ of $A$, there exist
$\dt>0,$   a finite subset ${\mathcal G}$ of $A$,
and a finite subset ${\mathcal P}$ of projections in $A$ with the following property.

Let $\psi,\phi: A\to Q$ be two unital ${\mathcal G}$-$\dt$-multiplicative completely positive maps such that
$$[\psi]|_{\mathcal P}=[\phi]|_{\mathcal P}.$$
Then there {are} a unitary $u\in Q$ and a  unital ${\mathcal F}$-$\ep$-multiplicative completely positive map $L: A\to {\rm C}([0,1], Q)$
such that
\begin{equation}\label{puni-2}
 \pi_0\circ L=\psi \quad\textrm{and}\quad \pi_1\circ L={\rm Ad}\, u\circ \phi.
\end{equation}
Moreover, if
\begin{equation}\label{puniq-3}
 |\mathrm{tr}\circ \psi(h)-\mathrm{tr}\circ \phi(h)|<{{\eps',}}\quad h\in {\mathcal H},
\end{equation}
for a finite set ${\mathcal H}\subseteq A$ and $\eps'>0$,
then $L$ may be chosen such that
\begin{equation}\label{puniq-4}
 |\mathrm{tr}\circ \pi_t\circ L(h)-\mathrm{tr}\circ \pi_0\circ L(h)|<\eps', \quad h\in {\mathcal H},\ t\in [0, 1].
\end{equation}
Here, $\pi_t: {\rm C}([0,1], Q)\to Q$ is the point evaluation at $t\in [0,1].$
\end{lem}

\begin{proof}
This is a direct application of the  stable uniqueness theorem (Corollary 2.6) of \cite{EN-K0-Z}, restated as Lemma \ref{stable-uniq-Q} above.
Let ${\mathcal F}\subseteq A$ be a finite subset and let $\ep>0$ be given.   We may assume that $1_A\in {\mathcal F}$ and every element of ${\mathcal F}$ has norm at most one.
Write  ${\mathcal F}_1=\{ab: a, b\in {\mathcal F}\}$. Note that $\mathcal F\subseteq\mathcal F_1$. 

Let  $\delta$, ${\mathcal G},$ ${\mathcal P},$ and $n$ be as assured by Lemma \ref{stable-uniq-Q} for ${\mathcal F}_1$ and $\ep/4.$ We may also assume that $\mathcal F\subseteq \mathcal G$ and $\delta\leq \eps/4$.

{{Let $\phi', \psi': A\to Q\otimes Q$ be defined by $\phi'(a)=\phi(a)\otimes 1$ and $\psi'(a)=\psi(a)\otimes 1$ for all $a\in A$.}}
Pick mutually equivalent projections
$e_0, e_1,e_2,...,e_n\in {{Q}}$ satisfying $\sum_{i=0}^n e_i=1_{{Q}}.$
Then, consider the maps $\phi_i, \psi_i: A\to {{Q}}\otimes e_i{{Q}}e_i$, $i=0, 1, ..., n$, which are defined by
$$\phi_i(a)=\phi(a)\otimes e_i\quad\mathrm{and}\quad \psi_i(a)=\psi(a)\otimes e_i,\quad a\in A,$$ 
and consider the {{finite sequence of}} maps {{from $A$ to $Q\otimes Q$}}
$$\Phi_{n+1}:={{\phi'}}=\phi_0\oplus \phi_1\oplus\cdots \oplus \phi_n,\quad {{\Phi_0:={{\psi'}}=\psi_0\oplus \psi_1\oplus \cdots \oplus \psi_n,}}$$ and
$$\Phi_i:=\phi_0\oplus\cdots  \oplus \phi_{i-1}\oplus \psi_{i}\oplus \cdots \oplus \psi_{n},\quad i=1,2,...,n.$$
Since $e_i$ is unitarily equivalent to $e_{{0}}$ for all $i$, one has  
$$[\phi_i]|_{\mathcal P}=[\psi_j]|_{\mathcal P},\quad 0\leq i, j\leq n.$$
and in particular, 
\beq\label{puniq-10+}
[\phi_{i}]|_{\mathcal P}=[\psi_{i}]|_{\mathcal P},\quad i=0, 1, .., n.
\eneq

Note that, for each $i=0, 1, ..., n$,
\begin{displaymath}
\begin{array}{lll}
\Phi_i &\sim& \psi_{i}\oplus (\phi_0\oplus \phi_1\oplus \cdots \oplus \phi_{i-1}\oplus  \psi_{i+1}\oplus  \psi_{i+2}\oplus \cdots \oplus \psi_n),
\\
\Phi_{i+1} &\sim& \phi_{i}\oplus (\phi_0\oplus \phi_1\oplus \cdots  \oplus \phi_{i-1}\oplus \psi_{i+1}\oplus  \psi_{i+2}\oplus \cdots \oplus \psi_n),
\\
\end{array}
\end{displaymath}
where $\sim$ denotes the relation of unitary equivalence. In view of this, and \eqref{puniq-10+} {{(identifying $Q\otimes Q$ with $Q$)}}, applying Lemma \ref{stable-uniq-Q} { {to $\phi:=\phi_i$, $\psi:=\psi_i$ and 
$$\xi:=(\phi_0\oplus \phi_1\oplus \cdots \oplus \phi_{i-1}\oplus  \psi_{i+1}\oplus  \psi_{i+2}\oplus \cdots \oplus \psi_n),$$}} we obtain unitaries {{$u_i\in Q\otimes Q$}}, $i=0, 1, ..., n$,  {{with $u_0=1_{Q\otimes Q}$}}
such that
$$\|\Phi_{i+1}(a)-{\rm Ad}u_{i+1}\circ \Phi_i(a)\|<\ep/4,\quad a\in {\mathcal F}_1.$$ Consequently,
\begin{equation}\label{puniq-11}
\|{\tilde \Phi}_{i+1}(a)-{\tilde \Phi}_{i}(a)\|<\ep/4,\quad a\in {\mathcal F}_1,
\end{equation}
where
$${\tilde \Phi}_0:=\Phi_0=\psi \quad\textrm{and}\quad {\tilde \Phi}_{i+1}:={\rm Ad}\,  u_{i}\circ \cdots \circ {\rm Ad}\ u_1 \circ   {\rm Ad}\ u_0\circ \Phi_{i+1},\quad i=0, 1,...,n.$$
Put $t_i=i/(n+1)$, $i=0, 1, ..., n+1$,
and define {{$L': A\to {\rm C}([0,1], Q\otimes Q)$}} by
$$
\pi_t\circ {{L'}}=(n+1)(t_{i+1}-t){\tilde \Phi_i}+(n+1)(t-t_i){\tilde \Phi_{i+1}},\quad t\in [t_i, t_{i+1}],\ i=0,1,...,n.
$$
By construction,
\begin{equation}\label{eq-verification}
\pi_0\circ {{L'}}={\tilde \Phi}_0={{\psi'}}\quad\mathrm{and}\quad \pi_1\circ {{L'}}={\tilde \Phi}_{n+1}={\rm Ad}\, u_n\circ\cdots \circ{ \rm Ad}\ u_1 \circ{ \rm Ad}\ u_0 \circ {{\phi'}}.
\end{equation}
Since $\tilde{\Phi}_i$, $i=0, 1, ..., n$, are $\mathcal G$-$\delta$-multiplicative (in particular $\mathcal F$-$\eps/4$-multiplicative), it follows from \eqref{puniq-11} that {{$L'$}} is ${\mathcal F}$-$\ep$-multiplicative.
{{Let $u'=u_{n}\cdots u_1 u_0$. Then $\pi_0\circ L'=\psi'$ and $\pi_1\circ L'={\rm Ad} u'\circ\psi'$.}}

{{Choose 
an isomorphism  $s: Q\otimes Q \to Q.$   Note that ${\rm tr} (s(a))={\rm tr}(a)$
for all $a\in Q\otimes Q.$
{{Recall that the endomorphism}} $s\circ j: Q\to Q$ is approximately unitary equivalent to ${{{\rm id}_Q}}: Q\to Q$, where $j: Q\to Q\otimes Q$ is defined by $j(x)=x\otimes 1$.
Thus there are two unitaries $w_0, w_1\in Q$ such that
\begin{equation}\label{2023-01-13}\|{\rm Ad}w_0\circ s\circ j(\psi(a))-\psi(a)\|<\ep/4~~\mbox{and}~~\|{\rm Ad}w_1\circ s\circ j(\phi(a))-\phi(a)\|<\ep/4,\end{equation}
for all {{$a\in {\cal F}_1$}}. Note that $j\circ \psi=\psi'$ and $j\circ \phi=\phi'$. Now define $L: A\to {\rm C}([0,1], Q)$ by 
{\small{
{{\begin{equation}\nonumber
L(t) = \left\{\begin{array}{ll}
(1-3t)\psi+ 3t{\rm Ad}w_0 \circ s\circ \psi', & t\in [0,1/3], \\
{\rm Ad} w_0\circ {\rm Ad} (s(u'))\circ s \circ L'(3(t-1/3)), & t\in [1/3,2/3] ,\\
3(t-2/3) {\rm Ad} w_0\circ {\rm Ad} (s(u')) \circ {\rm Ad} w_1^* \circ \phi + (3-3t){\rm Ad} w_0\circ {\rm Ad} (s(u')) \circ \phi' & t\in [2/3, 1].
\end{array}\right.
\end{equation}}}}}

\noindent
Finally let $u=w_0 (s(u'))w_1^*$;}} we have $\pi_0\circ L =\psi$ and $\pi_1\circ L={\rm Ad}u\circ \phi$.
%
It follows from \eqref{2023-01-13} and the choice of $\mathcal F_1$ that $L$ is $\mathcal F$-$\eps$-multiplicative ($L'$ is already $\mathcal F$-$\eps$ multiplicative).
Note that ${\rm tr}(\Phi_0(a))={\rm tr} (\psi'(a))={\rm tr}(\psi(a))$ for all $a\in A$. 
{{Suppose that \eqref{puniq-3} holds for some finite subset ${\cal H}$ and given $\ep'.$}}
From the definition
of $\Phi_i,$  we know 
\begin{equation}\label{2023-01-20}
\|{\rm tr}(\Phi_i(h))-{\rm tr}(\Phi_0(h))\|=\frac{i}{n+1}\|{\rm tr}(\phi(h))-{\rm tr}(\psi(h))\|<\ep'~~\mbox{for all} ~~h\in {\cal H}
.\end{equation}
It is then straightforward to verify that $L$ also satisfies \eqref{puniq-4}.
{{
In fact, if $\xi_0, \xi_1, \xi: A\to Q$\\
 ($=Q\otimes Q)$ are three linear maps satisfying  $\|{\rm tr}(\xi_i(h))-{\rm tr}(\xi(h))\|<\ep'$ ($i=0,1$) for all $h\in {\cal H}$, then  any convex combination $\xi':=t\xi_0+(1-t)\xi_1$ (where $0\leq t\leq 1$) also satisfies $\|{\rm tr}(\xi'(h))-{\rm tr}(\xi(h))\|<\ep'$ for $h\in {\cal H}$.}}
{{
Note that up to unitary equivalence, $L(t)$ (for $t\in [1/3, 2/3]$), is a convex combination of $s\circ \Phi_i$ and $s \circ \Phi_{i+1}$ (for suitable $i$). Hence for $t\in [1/3,2/3]$, 
$$\|{\rm tr}(L(t)(h))-{\rm tr}(L(0)(h))\|=\|{\rm tr}(L(t)(h))-{\rm tr}(\psi(h))\|
=\|{\rm tr}(L(t)(h))-{\rm tr}(\Phi_0(h))\|<\ep'$$ 
for $h\in {\cal H}$.
The inequality also holds for $t \in [0,1/3]$ (since ${\rm tr}(\psi(a))={\rm tr}(\psi'(a))$ for all $a\in A$) and for $t\in [2/3,1]$ (since ${\rm tr}(\phi(a))={\rm tr}(\phi'(a))$ for all $a\in A$). We obtain  \eqref{puniq-4}, as desired.}}

\end{proof}

\begin{lem}\label{T06}
Let $A$ be a unital simple separable  amenable 
C*-algebra
with ${\rm T}(A)={\rm T}_{\mathrm{qd}}(A)$
which satisfies the UCT. Assume that $A\otimes Q\cong A.$

For any $\sigma>0$, $\ep>0,$ and any finite subset
${\mathcal F}$ of $A$,
there exist a finite set of projections ${\mathcal P}$ in $A$ 
and $\dt>0$ with the following property.

Denote by $G\subseteq \Kzero(A)$ the subgroup generated by $\mathcal P\cup\{1_A\}$.
Let $\kappa: G\to \Kzero(C)$ be a positive homomorphism
with $\kappa([1_A])=[1_C]$, where $C={\rm C}([0,1], Q)$, and let $\lambda: {\rm T}(C)\to {\rm T}(A)$ be a continuous affine map
such that
\begin{equation}\label{07-1-1}
|{{\tau(\kappa([p]))}} - \lambda(\tau)(p)|<\dt,\quad  p\in {\mathcal P},\ \tau\in\mathrm{T}(C).
\end{equation}
(In particular, this entails that $\mathrm{T}(A)\neq\O.$) 
Then
there is a unital  ${\mathcal F}$-$\ep$-multiplicative completely positive map  $L: A\to C$ such that
\begin{equation}\label{07-1}
| \tau\circ L(a)-\lambda(\tau)(a)|<\sigma,\quad a\in\mathcal F,\ \tau\in {\rm T}(C).
\end{equation}
\end{lem}

\begin{proof}
Let $\ep,$ $\sigma$ and ${\mathcal F}$ be given. We may assume that every element of $\mathcal F$ has norm at most one.

Let $\delta_1$ (in place of $\delta$), $\mathcal G$, and $\mathcal P$ be as assured by Lemma \ref{partuniq} for $\mathcal F$ and $\eps$. 
We may assume
that ${\cal F\cup \mathcal P}\subseteq {\cal G}.$


Adjoining $1_A$ to $\mathcal P$, write ${\mathcal P}=\{1_A, p_1, p_2,...,p_s\}.$ Deleting one or more of $p_1, p_2, ..., p_s$ (but not $1_A$), we may assume that the set
$\{[1_A], [p_1],...,[p_s]\}$ is $\Q$-linearly independent. {{(Since $A\cong A\otimes Q$, we have $K_0(A)\cong K_0(A)\otimes \Q$.)}}

Let $\dt_2>0$ (in place of $\dt$) be as assured by Lemma \ref{0.5} for
$\ep_1=\dt_1,$ $\ep_2=\sigma/4,$ ${\mathcal G}$, and $\{p_1,p_2,...,p_s\}.$

Put $\dt=\min\{\dt_1, {\dt_2/8}, 1/4\}$, and let us show that $\mathcal P$ and $\delta$ are as desired.

Let $\kappa$ and $\lambda$ be given satisfying \eqref{07-1-1}. 

Let $\lambda_*: \Aff({\rm T}(A))\to \Aff({\rm T}(C))$ be defined by
$\lambda_*(f)(\tau)=f(\lambda(\tau))$ for all $f\in \Aff({\rm T}(A))$ and $\tau\in {\rm T}(C).$
Identify $\partial_e{({\rm T}(C))}$ with $[0,1]$ {{(that is, identify  $\mathrm{tr}\circ\pi_t$ with $t$, where $\pi_t:  C={\rm C}([0,1], Q)\to Q$ is the point evaluation at $t\in [0,1]$)}}, and put $\eta=\min\{\dt, \sigma/12\}.$
 Choose a partition
$$
0=t_0<t_1<t_2<\cdots <t_{n-1}<t_n=1
$$
of the interval $[0, 1]$ such that 
\begin{equation}\label{T06-2}
|\lambda_*(\hat{g})(t_j) - \lambda_*(\hat{g})(t_{j-1})| < \eta,\quad g\in {\mathcal G},\  j=1,2,...,n.
\end{equation}
{{(Here recall that $\hat{g}\in \Aff({\rm T}(A))$ is given by $\hat{g}(\tau)=\tau (g)$ for any $\tau\in {\rm T}(A)$.)}}

Since ${\rm T}(A)={\rm T}_{{\rm qd}}(A),$  there are unital ${\mathcal G}$-$\dt$-multiplicative completely positive  maps
$\Psi_j:  A\to Q$, $j=0, 1, 2,  ..., n$, such that
\begin{equation}\label{T06-3}
|\mathrm{tr}\circ \Psi_j(g)-\lambda_*(\hat{g})(t_j)|<\eta,\quad g\in {\mathcal G}.
\end{equation}
It then follows from \eqref{pert-proj}, \eqref{T06-3}, and \eqref{07-1-1}  that, for each $i=1,2,...,s$ and each $j=1, 2, ..., n$,
\beq\label{T06-4}
 |{\rm tr}([\Psi_j(p_i)])-{\rm tr}([\Psi_0(p_i)])| &<& 4\delta +2\eta +|\lambda_*(\hat{p_i})(t_j) - \lambda_*(\hat{p_i})(t_{0})| \nonumber\\
&<&4\delta +2\eta+2\dt+|{\rm tr}\circ \pi_{t_j}(\kappa([p_i]))-{\rm tr}\circ\pi_0(\kappa([p_i])| \nonumber \\
&=& 2\eta+6\dt\leq 8\delta\leq \delta_2.
\eneq
(Here, as before, $\pi_t$ is the point evaluation at $t\in [0,1]$.)
We also have, by (\ref{T06-2}) and (\ref{T06-3}), that
\beq\label{T06-4+n}
|\mathrm{tr}(\Psi_j(g))-\mathrm{tr}(\Psi_{j+1}(g))|<3\eta,\quad  g\in {\cal G},\ j=1, 2, ..., n.
\eneq
Consider the differences
\begin{equation}\label{1228-1}
r_{i,j}:={\rm tr}([\Psi_j(p_i)])-{\rm tr}([\Psi_0(p_i)]),\quad i=1,2,...,s,\ j=1,2, ...,n.
\end{equation}
By \eqref{T06-4}, $r_{i, j}<\delta_2$. Applying Lemma \ref{0.5}, we obtain a projection $e\in Q$ with $\mathrm{tr}(e)<\sigma/4$ and ${\mathcal G}$-$\dt_1$-multiplicative  unital completely positive maps
$\psi_0, \psi_j: A\to eQe$, $j=1,2,...,n,$ such that
\beq\label{T06-5-NN}
[\psi_0(p_i)]-[\psi_j(p_i)]=r_{i,j},\quad i=1,2,...,s,\  j=1,2,...,n.
\eneq

Consider the direct sum maps
$$\Phi_j':=\psi_j\oplus \Psi_j: A\to ({{e\oplus 1}})\mathrm{M}_2(Q)({{e\oplus 1}}),\quad j=0, 1,2,...,n.$$
Since $\delta\leq \delta_1$, these are $\mathcal G$-$\delta_1$-multiplicative.  By  (\ref{1228-1}) and (\ref{T06-5-NN}),
\beq\label{T06-7}
[\Phi_j'(p_i)]=[\Phi_0'(p_i)],\quad i=1,2,...,s,\  j=1,2,...,n.
\eneq
Define $s: \Q\to \Q$ by $s(x)=x/(1+\mathrm{tr}(e))$, $x\in \Q.$
Choose a (unital) isomorphism $$S: ({{e\oplus 1}})\mathrm{M}_2(Q)({{e\oplus 1}})\to Q$$ such
that $S_{*0}=s.$

Consider the composed maps, still $\mathcal G$-$\delta_1$-multiplicative, and now unital,
$$\Phi_j:=S\circ \Phi_j': A\to Q, \quad j=0,1, 2, ..., n.$$
By (\ref{T06-7}),
$$[\Phi_j]|_{\mathcal P}=[\Phi_{j-1}]|_{\mathcal P},\quad j=1, 2, ..., n,$$
and by (\ref{T06-4+n}) and the fact that $\mathrm{tr}(e)<\sigma/4,$
\begin{equation}\label{sm-tr-ed-pts}
|\mathrm{tr}\circ\Phi_j(a)-\mathrm{tr}\circ\Phi_{j-1}(a)|<3\eta+\sigma/4\leq \sigma/2,\quad a\in {\mathcal F},\ j=1, 2, ...,n.
\end{equation}

It follows by Lemma \ref{partuniq}, applied successively for $j=1, 2, ..., n$ (to the pairs $(\Phi_0, \Phi_1)$, $(\mathrm{Ad}\, u_1\circ\Phi_1, \mathrm{Ad}\, u_1\circ\Phi_2)$, ..., $(\mathrm{Ad}\, u_{n-1}\circ\cdots\circ\mathrm{Ad}\, u_1\circ\Phi_{n-1}, \mathrm{Ad}\, u_{n-1}\circ\cdots\circ\mathrm{Ad}\, u_1\circ\Phi_{n})$), that there are, for each $j$, a unitary $u_j\in Q$ and a unital
${\mathcal F}$-$\ep$-multiplicative completely positive map $L_j: A\to {\rm C}([t_{j-1}, t_{j}], Q)$ such that
\begin{equation}\label{prop-1}
\pi_0\circ L_1=\Phi_0, \quad \pi_{t_1}\circ L_1={\rm Ad}\, u_1\circ \Phi_1,
\end{equation}
and
\begin{equation}\label{prop-2}
\pi_{t_{j-1}}\circ L_j=\pi_{t_{j-1}}\circ L_{j-1},\quad \pi_{t_{j}}\circ L_j={\rm Ad}\, u_j\circ\cdots\circ\mathrm{Ad}\, u_1\circ \Phi_j, \quad j=2, 3, ..., n.
\end{equation}
Furthermore, {{applying the ``moreover" part of Lemma \ref{partuniq},}} in view of \eqref{sm-tr-ed-pts}, we may choose the maps $L_j$ such that
\begin{equation}\label{prop-tr}
|{\rm tr}\circ \pi_t\circ L_j(a)-\lambda({\rm tr}\circ \pi_t)(a)|<\sigma,\quad  t\in [t_{j-1}, t_{j}],\ a\in\mathcal F,\ j=1, 2, ..., n.
\end{equation}

Define $L: A\to {\rm C}([0,1], Q)$ by
$$\pi_t \circ L= \pi_t \circ L_j,\quad t\in [t_{j-1}, t_{j}],\ j=1,2,...,n.$$
Since $L_j$, $j=1, 2, ..., n$, are $\mathcal F$-$\eps$-multiplicative (use \eqref{prop-1} and \eqref{prop-2}), we have that $L$ is a unital $\mathcal F$-$\eps$-multiplicative completely positive map $A \to \mathrm{C}([0, 1], Q)$. (Note that the construction of the map $L$ is different from---is based on---the construction of the map $L$ in the proof of Lemma \ref{partuniq}.) It follows from \eqref{prop-tr} that $L$ satisfies \eqref{07-1}, as desired.
\end{proof}

\begin{thm}\label{TT}
Let $A$ be a unital simple separable C*-algebra with finite nuclear dimension. Assume that
${\rm T}(A)={\rm T}_{{\rm qd}}(A)\not=\O$ and that $A$ satisfies the UCT.
Then ${\mathrm{gTR}}(A\otimes Q)\le 1,$ and so (if $A$ is not elementary), $A\in {\mathcal N}_1.$
\end{thm}

\begin{proof}
Since $A$ is simple, the assumption
${\rm T}(A)={\rm T}_{{\rm qd}}(A)\not=\O$
immediately implies that  $A$ is both stably finite and quasidiagonal.
%
Since $A$ is unital, simple, separable, and has finite nuclear dimension, by \cite{W}, it is $\mathcal Z$-stable. By the definition of $\mathcal N_1$, it remains to show that $\mathrm{gTR}(A\otimes Q) \leq 1$. To prove that $\mathrm{gTR}(A\otimes Q) \leq 1$, we may assume that $A\otimes Q \cong A$.
%
With this assumption, by \cite{R},
$A$ has stable rank one.

By \cite{Ell} (see also Corollary 13.51 of \cite{GLNI}), together with the assumption $A\cong A\otimes Q$, there is a unital simple C*-algebra $C=\lim_{n\to\infty}(C_n, \imath_n),$
where each $C_n$ is the tensor product of a C*-algebra in $\mathcal C_0$ with $Q$ 
and $\imath_n$ is injective, such that
$$
(\Kzero(A), \Kzero(A)_+, [1_A]_0, {\rm T}(A), r_A)\cong (\Kzero(C), \Kzero(C)_+, [1_C]_0, {\rm T}(C), r_C).
$$
Choose
an isomorphism  $\Gamma$ as above, and write
$\Gamma_\aff$ for the corresponding map from $\Aff({\rm T}(A))$ to $\Aff({\rm T}(C)).$

Let a finite subset  ${\mathcal F}$ of $ A$ and $\ep>0$ be given.


Let the finite set $\mathcal P$ of projections in $A$, the finite subset $\mathcal G$ of $A$, and
$\dt>0$ be as assured by Lemma \ref{partuniq} for ${\mathcal F}$ and $\ep$.
 We may assume that $1_A\in {\mathcal P}.$
Write  $\mathcal P=\{1_A, p_1,p_2,...,p_s\}.$  Deleting some elements (but not $1_A$), we may assume that the set
$$\{[1_A], [p_1], [p_2],...,[p_s]\}\subseteq \Kzero(A)$$ is $\Q$-linearly independent. {(Recall $K_0(A)\cong K_0(A)\otimes \Q$ as $A\cong A\otimes Q$.)}

We may also assume, \wilog, that ${\mathcal F}\cup {\mathcal P}\subseteq {\mathcal G}$, $\dt\leq \ep$, and every element of $\mathcal G$ has norm at most one. 

Let $\sigma>0.$
Let $\dt_1>0$ (in place of $\delta$) be as assured by Lemma \ref{0.6} for ${\mathcal G}$, $\dt$ (in place of $\ep_1$), and
$\sigma/64$ (in place of $\ep_2$). We may assume that $\delta_1\leq 8 \delta$.


Let $\dt_3>0$ (in place of $\dt$) be as assured by Lemma \ref{0.7}  for ${\cal G},$ $\dt_1/8$ (in place of $\ep_1$),
and
$\min\{\dt_1/32,  \sigma/256\}$ (in place of $\ep_2$).

Let ${\cal P}_1$ (in place of ${\cal P}$) and
$\dt_2>0$ (in place of $\dt$)  be as assured by Lemma \ref{T06} for
$\dt_1/8$ (in place of $\ep$),
$\min\{\dt_1/32,\sigma/256\}$ (in place of $\sigma$), and ${\cal G}$
(in place of ${\cal F}$).
Replacing $\mathcal P$ and $\mathcal P_1$ by their union, we may assume that ${\cal P}={\cal P}_1.$


By Lemma 2.9 of \cite{EN-K0-Z},
 there is a unital positive linear map
$$
\gamma: \Aff({\rm T}(A))\to \Aff({\rm T}(C_{n_1}))
$$
for some $n_1\ge 1$ such that
\begin{equation}\label{TT-1}
\|(\imath_{n_1, \infty})_\aff\circ\gamma(\hat{f})-\Gamma_\aff(\hat{f})\|<\min\{77\sigma/128, \delta_2, \delta_3/2\},\quad f\in {\mathcal F}\cup {\mathcal P}.
\end{equation}


 We may assume, \wilog,
that there are projections  $p_1',p_2',...,p_s'\in C_{n_1}$ such that
$\Gamma([p_i])=\imath_{n_1,\infty}([p_i']),$ $i=1,2,...,s.$
To simplify notation, assume that
$n_1=1.$
Let $G_0$ denote the subgroup of $\Kzero(A)$ generated by ${\mathcal P}.$
{{Since $G_0$ is  generated freely  by $[1_A], p_i,i=1,2,\cdots,s,$  we can define $\Gamma': G_0\to \Kzero(C_1)$ by
\begin{equation}\label{lft-proj}
\Gamma'([1_A])=[1_{C_1}], ~~\Gamma'([p_i])=[p_i'],\quad i=1,2,...,s.
\end{equation} Hence
\begin{equation*}
(\imath_{1, \infty})_{*0}\circ \Gamma'=\Gamma|_{G_0}.
\end{equation*} }}
%
%
Since the pair $(\Gamma_\aff, \Gamma|_{\Kzero(A)})$ is compatible, as a consequence of \eqref{TT-1} and \eqref{lft-proj} we have
\begin{equation}\label{TT-3}
\|\widehat{p_i'}-\gamma(\widehat{p_i})\|_{\infty}<\min\{\delta_2, \delta_3/2\},
\quad i=1,2,....,s.
\end{equation}
Write
$$C_1=(\psi_0, \psi_1, Q^r, Q^l)=\{(f,a)\in {\rm C}([0,1], Q^r)\oplus Q^l: f(0)=\psi_0(a)\andeqn f(1)=\psi_1(a)\},$$
where $\psi_0, \psi_1: Q^l\to Q^r$ are unital homomorphisms.

%
Denote by $$\pi_{\mathrm{e}}: C_1 \to Q^l,\ (f, a)\mapsto a$$
the canonical quotient map, and by  $j: C_1\to {\rm C}([0,1], Q^r)$ the canonical map
$$j((f,a))=f, \quad (f,a)\in C_1\otimes Q.$$
Denote by $\gamma^*: {\rm T}(C_1)\to {\rm T}(A)$ the continuous affine map
dual to $\gamma.$
Denote by  $\theta_1,\theta_2,...,\theta_l$  the extreme tracial states of $C_1$ factoring through $\pi_{\mathrm{e}}: C_1 \to Q^l$.

By the assumption ${\rm T}(A)={\rm T}_{\mathrm {qd}}(A),$ there exists a unital
${\mathcal G}$-$\min\{\dt_1/8, \dt_3/8\}$-multiplicative completely positive map $\Phi: A\to Q^l$ such that
\begin{equation}\label{TT-4}
|\mathrm{tr}_j\circ \Phi(a)-\gamma^*(\theta_j)(a)|< \min\{13\delta_1/32, \delta_3/4, \sigma/32\},\quad a\in {\mathcal G},\ j=1,2,..., l,
\end{equation}
where $\mathrm{tr}_j$ is the tracial state supported on the $j$th direct summand of $Q^l$.
Moreover, since ${\cal P}\subseteq {\cal G},$ as in (\ref{pert-proj}), we also have that
\begin{equation}\label{pert-p}
|\mathrm{tr}_j([\Phi(p_i)]) - \mathrm{tr}_j(\Phi(p_i))| < \delta_3/4,\quad i=1, 2, ..., s,\ j=1, 2, ..., l.
\end{equation}

Set
$$\D:=(\pi_{\mathrm{e}})_{*0}(\Kzero(C_1))=\ker((\psi_0)_{*0}-(\psi_1)_{*0})\subseteq \Q^{l}.$$
It follows from 
\eqref{TT-4} that 
\begin{equation}\label{trace-infty}
|\tau(\Phi(a)) -  (\pi_{\mathrm{e}})_\aff(\gamma (\hat{a}))(\tau)| < \min\{13\delta_1/32, \delta_3/4, \sigma/32\},\quad a\in \mathcal G,\ \tau\in \mathrm{T}(Q^l),
\end{equation}
where $(\pi_{\mathrm{e}})_{\aff}: \aff(\mathrm{T}(C_1)) \to \aff(\mathrm{T}(Q^l))$ is the map induced by $\pi_{\mathrm{e}}$.
By \eqref{trace-infty} for $a\in\mathcal P\subseteq\mathcal G$, together with \eqref{pert-p} and \eqref{TT-3},
\begin{equation*}
|\tau([\Phi(p_i)]) - \tau\circ (\pi_{\mathrm{e}})_{*0}\circ \Gamma'([p_i])|<\delta_3,\quad\tau\in {\rm T}(Q^l),\ i=1, 2, ...,s.
\end{equation*}

Therefore, applying Lemma \ref{0.7}, with $r_i=[\Phi(p_i)]-(\pi_{\mathrm{e}})_{*0}\circ \Gamma'([p_i])$, we obtain ${\mathcal G}$-$\dt_1/8$-multiplicative completely positive maps
$\Sigma_1, \mu_1 : A\to Q^l$, with ${{\Sigma_1}}(1_A)=\mu_1(1_A)$ a projection, such that
\begin{equation}\label{pert-1-1}
 \tau(\Sigma_1(1_A))<\min\{\dt_1/32,\sigma/256\},\quad \tau\in {\rm T}(Q^l),
\end{equation}
\begin{equation}\label{pert-1-2}
[\Sigma_1({\mathcal P})]\subseteq \D,\quad\mathrm{and}
\end{equation}
\begin{equation} \label{pert-1-3}
{[}\Sigma_1(p_i){]}-[\mu_1(p_i)]=[\Phi(p_i)]-(\pi_{\mathrm{e}})_{*0}\circ \Gamma'([p_i]),\quad  i=1,2,...,s.
\end{equation}

Consider the (unital) direct sum map
\begin{equation}\label{defn-Phi-P}
\Phi':=\Phi\oplus\mu_1: A\to (1\oplus\Sigma_1(1_A))\mathrm{M}_2(Q^l)(1\oplus\Sigma_1(1_A)).
\end{equation}
Note that $\Phi'$, like $\mu_1$ and $\Phi$, is $\mathcal G$-$\delta_1/8$-multiplicative. 
It follows from \eqref{pert-1-3} that for each $i=1, 2, ..., s,$
\begin{equation}\label{TT-7-1}
{{(\psi_0)_{*0}([\Phi'(p_i)])}} 
=(\psi_0)_{*0}([\mu_1(p_i)]+[\Phi(p_i)])=(\psi_0)_{*0}({[}\Sigma_1(p_i){]}+(\pi_{\mathrm{e}})_{*0}\circ \Gamma'([p_i]))
\end{equation}
and
\begin{equation}\label{TT-7-2}
{{(\psi_1)_{*0}([\Phi'(p_i)])}}= (\psi_1)_{*0}([\mu_1(p_i)]+[\Phi(p_i)])=(\psi_1)_{*0}({[}\Sigma_1(p_i){]}+(\pi_{\mathrm{e}})_{*0}\circ \Gamma'([p_i])).
\end{equation}
It follows from \eqref{TT-7-1} and \eqref{TT-7-2}, in view of \eqref{pert-1-2} and the fact (use \eqref{lft-proj}) that $(\pi_{\mathrm{e}})_{*0}\circ \Gamma'([p_i]))\in (\pi_{\mathrm{e}})_{*0}(\Kzero(C_1))=\D,$
that $[\Phi'(p_i)]\in\mathbb D$, $i=1, 2, ..., s$, i.e.,
\beq\label{TT-8-nouse}
{{(\psi_0)_{*0}([\Phi'(p_i)])={{(\psi_1)_{*0}([\Phi'(p_i)])}}}}\quad i=1,2,...,s.
\eneq

Set $B={\rm C}([0,1], Q^r),$ and (as before) write  $\pi_t: B\to  Q^r$ for  the point evaluation at $t\in [0,1].$
Since $1_A\in\mathcal P$, by \eqref{pert-1-2}, $[\Sigma_1(1_A)]\in\mathbb D$, and so there is
a projection $e_0\in B$ such that $\pi_0(e_0)=\psi_0(\Sigma_1(1_A))$ and $\pi_1(e_0)=\psi_1(\Sigma_1(1_A)).$
It then follows from \eqref{pert-1-1} (applied just for $\tau$ factoring through $\psi_0$---alternatively, for $\tau$ factoring through $\psi_1$) that
\begin{equation}\label{small-trace-n-2}
\tau(e_0) < \min\{\dt_1/32,\sigma/256\},\quad \tau\in\mathrm{T}(B).
\end{equation}

Let $j^*: {\rm T}(B)\to {\rm T}(C_1)$ denote the continuous affine map dual to the canonical unital map $j: C_1\to B.$
Let $\gamma_1: \mathrm{T}(B)\to\mathrm{T}(A)$ be defined by $\gamma_1:=\gamma^*\circ j^*$, and let $\kappa: G_0\to \Kzero(B)$
be defined by $\kappa:=j_{*0}\circ \Gamma'.$
Then, by (\ref{lft-proj}) and (\ref{TT-3}),
\beq\label{TT-8+}
|\tau(\kappa([p_i]) - \gamma_1(\tau)(p_i)| & = & | \tau(j_{*0}(\Gamma'([p_i])) - (\gamma^*\circ j^*)(\tau)(p_i)| \nonumber \\
 & = & |j^*(\tau)([p_i']) - \gamma(\widehat{p_i})(j^*(\tau)) |<\dt_2
\eneq
for all $\tau\in {\rm T}(B),$ {{$i=1,2,...,s.$}}


The estimate \eqref{TT-8+} ensures that we can apply Lemma \ref{T06} with $\kappa$ and $\gamma_1$ (note that $\Gamma'([1_A])=[1_{C_1}]$ and hence $\kappa([1_A])=[1_{B}]$), to obtain a unital ${\mathcal G}$-$\dt_1/8$-multiplicative completely positive map
$\Psi': A\to B$ such that
\begin{equation}\label{TT-9-nouse-1}
|\tau\circ \Psi'(a)-\gamma_1(\tau)(a)|<\min\{\dt_1/32, \sigma/256\},\quad a\in {\mathcal G},\ \tau\in {\rm T}(B).
\end{equation}
Amplifying $\Psi'$ slightly (by first identifying $Q^r$ with $Q^r\otimes Q$ and then considering
$H_0(f)(t)=f(t)\otimes(1+e_0(t))$ for $t\in[0, 1]$), we obtain a unital $\mathcal G$-$\delta_1/8$-multiplicative  completely positive map $\Psi: A \to (1\oplus e_0)\mathrm{M}_2(B)(1\oplus e_0)$
such that (by \eqref{TT-9-nouse-1} and \eqref{small-trace-n-2})
\begin{equation}\label{TT-9-use-n-1}
|\tau\circ \Psi(a)-\gamma_1(\tau)(a)| < 2\min\{\dt_1/32,\sigma/256\} =\min\{\dt_1/16, \sigma/128\}.
\end{equation}


Note that for any element $a \in C_1$,
\begin{equation}\label{const-trace}
\tau(\psi_0(\pi_{\mathrm{e}}(a)))=\tau(\pi_0(j(a)))\quad\mathrm{and}\quad \tau(\psi_1(\pi_{\mathrm{e}}(a))) = \tau(\pi_1(j(a))),\quad \tau\in\mathrm{T}(Q^r).
\end{equation}
(Recall that $j: C_1 \to B$ is the canonical map.)
Therefore (by \eqref{const-trace}), for any $a\in\mathcal G$ and $\tau\in\mathrm{T}(Q^r)$,
\begin{eqnarray}\label{trace-infty-1}
| \tau(\psi_0(\Phi(a))) - \gamma(\hat{a})(\tau\circ\pi_0\circ j) |&=&| \tau(\psi_0(\Phi(a))) -  \gamma(\hat{a})(\tau\circ\psi_0\circ\pi_{\mathrm{e}}) | \nonumber \\
&=& | \tau\circ\psi_0(\Phi(a))) - (\pi_{\mathrm{e}})_\aff(\gamma(\hat{a}))(\tau\circ\psi_0)| \nonumber \\
& < & \min\{13\delta_1/32, \sigma/32\} \quad\quad \textrm{(by \eqref{trace-infty})}.
\end{eqnarray}
The same argument shows that
\begin{equation}\label{trace-infty-1-1}
| \tau(\psi_1(\Phi(a))) - \gamma(\hat{a})(\tau\circ\pi_1\circ j) |<\min\{13\delta_1/32, \sigma/32\},\quad a\in\mathcal G,\ \tau\in\mathrm{T}(Q^r).
\end{equation}


Then, for any $\tau\in\mathrm{T}(Q^r)$ and any $a\in\mathcal G$, we have
\begin{eqnarray}\label{trace-mt-boundary}
&&\hspace{-0.2in} |\tau\circ \psi_0\circ \Phi'(a)-\tau\circ \pi_0\circ \Psi(a)| \nonumber \\
&& = | \tau\circ\psi_0(\Phi(a)\oplus\mu_1(a)) -  \tau\circ \pi_0\circ \Psi(a) | \nonumber \\
&&<
|\tau\circ \psi_0(\Phi(a)\oplus\mu_1(a))-\gamma_1(\tau\circ\pi_0)(a)| +\min\{\dt_1/16,\sigma/128\}
 \quad\textrm{(by \eqref{TT-9-use-n-1})}\nonumber \\
&& <  |\tau\circ \psi_0(\Phi(a))-\gamma_1(\tau\circ\pi_0)(a)| + \min\{3\dt_1/32, 3\sigma/256\}
\,\,\,\quad\quad\quad\textrm{(by  {{{\eqref{pert-1-1}}}})}\nonumber \\ \nonumber
&&= |\tau\circ \psi_0(\Phi(a))-  \gamma(\hat{a})(\tau\circ\pi_0\circ j) | + \min\{3\dt_1/32, 3\sigma/256\}  \\\label{trace-mt-boundary-n-1}
&&<  \min\{13\delta_1/32, \sigma/32\} + \min\{3\dt_1/32, 3\sigma/256\} \\
&& \leq 13\delta_1/32 + 3\delta_1/32= \delta_1/2 \quad\quad\quad\quad\quad\quad\quad\quad\quad\quad\quad\quad\quad\quad\quad\quad \textrm{(by \eqref{trace-infty-1})} \nonumber.
\end{eqnarray}

The same argument, using \eqref{trace-infty-1-1} instead of \eqref{trace-infty-1}, shows that
\begin{eqnarray}\label{trace-mt-boundary-n-1-1}
 |\tau\circ \psi_1\circ \Phi'(a)-\tau\circ \pi_1\circ \Psi(a)| & < & \min\{13\delta_1/32, \sigma/32\} + \min\{3\dt_1/32, 3\sigma/256\} \\
 & \leq & \delta_1/2,\quad \tau\in\mathrm{T}(Q^r),\ a\in\mathcal G. \nonumber
\end{eqnarray}
(\eqref{trace-mt-boundary-n-1} and \eqref{trace-mt-boundary-n-1-1}---the $\sigma$ estimates---will be used later to verify \eqref{TT-13+1} and (\ref{TT-13+2}).)

Noting that $\Psi$ and $\Phi'$ are $\delta_1/8$-multiplicative on $\{1_A, p_1, p_2, ..., p_s\}$, by our convention (see \eqref{pert-proj}),  we have, for all $\tau\in {\rm T}(Q^r),$
$$ |\tau([\Psi(p_i)]) - \tau(\Psi(p_i)) |<\delta_1/4\quad\mathrm{and}\quad |\tau([\Phi'(p_i)]) - \tau(\Phi'(p_i)) |<\delta_1/4,\quad i=1, 2, ..., s.$$
Combining these inequalities with  \eqref{trace-mt-boundary-n-1} and \eqref{trace-mt-boundary-n-1-1},
we have
\beq\label{TT-11}
|\tau([\pi_0\circ \Psi(p_i)]) - \tau([\psi_0\circ \Phi'(p_i)])|<\dt_1,\quad i=1,2,...,s,\ \tau\in\mathrm{T}(Q^r).
\eneq
Therefore (in view of \eqref{TT-11}), applying Lemma \ref{0.6} with $r_i=[\pi_0\circ \Psi(p_i)] - [\psi_0\circ \Phi'(p_i)]$, we obtain
${\mathcal G}$-$\dt$-multiplicative completely positive maps
$\Sigma_2: A\to Q^l$ and $\mu_2: A\to Q^r$, taking $1_A$ into projections,
such that
\begin{equation}\label{prop-2-n-1}
[\psi_k\circ \Sigma_2(1_A)]=[\mu_2(1_A)],\quad k=0, 1,
\end{equation}
\begin{equation}\label{prop-2-n-2}
{[} \Sigma_2(\mathcal P){]}\subseteq  (\pi_{\mathrm{e}})_{*0}(\Kzero(C_1))=\D,\quad i=1, 2, ..., s,
\end{equation}
\begin{equation}\label{prop-2-3}
 \tau(\Sigma_2(1_A))< \sigma/64,\quad \tau \in {\rm T}(Q^l), 
\end{equation}
and, taking into account \eqref{prop-2-n-2},
\begin{equation}\label{prop-2-4}
{[}\psi_0\circ \Sigma_2(p_i){]}-[\mu_2(p_i)]  =  {[}\psi_1\circ \Sigma_2(p_i){]}-[\mu_2(p_i)]
 =  [\pi_0\circ \Psi(p_i)]-[\psi_0\circ \Phi'(p_i)],
\end{equation}
where $i=1, 2, ..., s.$  {{By (\ref{prop-2-n-1}), there are unitaries $w_k, k=0,1$ such that
$$\psi_k\circ \Sigma_2(1_A)={\rm Ad} w_k\circ \mu_2(1_A),~k=0,1.$$
Let $\{w(t)\}_{0\leq t\leq 1}$
be a continuous path of {{unitaries}} in {{$Q^r$}}
such that $w(0)=w_0$  and $w(1)=w_1$.}}

Consider the four $\mathcal G$-$\delta$-multiplicative direct sum maps (note that $\Phi'$ and $\Psi$ are $\mathcal G$-$\delta_1/8$-multiplicative, and $\delta_1\leq 8\delta$), from $A$ to $\mathrm{M}_3(Q^r)$,
\beq\nonumber
&&\Phi_0:=(\psi_0\circ \Phi')\oplus (\psi_0\circ \Sigma_2),\quad \Phi_1:=(\psi_1\circ \Phi')\oplus (\psi_1\circ \Sigma_2)\andeqn\\
&&\Psi_0:=(\pi_0\circ \Psi) \oplus {{{\rm Ad}w_0\circ \mu_2}},\quad\quad \quad\,\,\,\,\,\,\Psi_1:=(\pi_1\circ \Psi) \oplus {{{\rm Ad}w_1\circ \mu_2}}.
\eneq

We then have that for each $i=1, 2, ..., s$,
\begin{eqnarray*}
[\Psi_0(p_i)]-[\Phi_0(p_i)] & = & ([(\pi_0\circ \Psi)(p_i)] + [\mu_2(p_i)]) - ([(\psi_0\circ \Phi')(p_i)] + [(\psi_0\circ \Sigma_2)(p_i)])  \\
& = & ([(\pi_0\circ \Psi)(p_i)] - [(\psi_0\circ \Phi')(p_i)] ) - ([(\psi_0\circ \Sigma_2)(p_i)] - [\mu_2(p_i)] )\\
& =& 0 \quad\quad \textrm{(by \eqref{prop-2-4})},
\end{eqnarray*}
and
\begin{eqnarray*}
&& [\Psi_1(p_i)]-[\Phi_1(p_i)] \\
& = & ([(\pi_1\circ \Psi)(p_i)] + [\mu_2(p_i)]) - ([(\psi_1\circ \Phi')(p_i)] + [(\psi_1\circ \Sigma_2)(p_i)]) \\
& = & ([(\pi_1\circ \Psi)(p_i)] - [(\psi_1\circ \Phi')(p_i)] ) - ([(\psi_1\circ \Sigma_2)(p_i)] - [\mu_2(p_i)] ) \\
& = & ([(\pi_0\circ \Psi)(p_i)] - [(\psi_1\circ \Phi')(p_i)] ) - ([(\psi_1\circ \Sigma_2)(p_i)] - [\mu_2(p_i)] )\quad \textrm{($\pi_0$ and $\pi_1$ are homotopic)} \\
&=& ([(\pi_0\circ \Psi)(p_i)] - [(\psi_0\circ \Phi')(p_i)] ) - ([(\psi_1\circ \Sigma_2)(p_i)] - [\mu_2(p_i)] )=0
\quad \textrm{(by \eqref{TT-8-nouse})} \\
& = & 0 \quad\quad   \textrm{(by \eqref{prop-2-4})}.
\end{eqnarray*}

Summarizing the calculations in the preceding paragraph, we have
\begin{equation}\label{TT-13}
[\Phi_i]|_{\mathcal P}=[\Psi_i]|_{\mathcal P},\quad i=0,1.
\end{equation}

On the other hand, for any $a\in\mathcal F\subseteq\mathcal G$ and any $\tau\in\mathrm{T}(Q^r)$, we have
\beq\nonumber
|\tau(\Phi_0(a)) - \tau(\Psi_0(a))| & = & |\tau ((\psi_0\circ \Phi')(a)\oplus (\psi_0\circ \Sigma_2)(a)) - \tau((\pi_0\circ \Psi)(a) \oplus \mu_2(a))| \\\nonumber
 & < & |\tau ((\psi_0\circ \Phi')(a)) - \tau((\pi_0\circ \Psi)(a) )| + \sigma/32\quad\quad\textrm{(by \eqref{prop-2-3})} \\\nonumber
 & < & \min\{13\delta_1/16, \sigma/32\} + \min\{3\dt_1/16, 3\sigma/256\} +\sigma/32  \quad\quad\textrm{(by \eqref{trace-mt-boundary-n-1})}\\\label{TT-13+1}
 &\leq & 5\sigma/64.
\eneq
The same argument, using \eqref{trace-mt-boundary-n-1-1} instead of \eqref{trace-mt-boundary-n-1},  also shows that
\beq\label{TT-13+2}
|\tau(\Phi_1(a)) - \tau(\Psi_1(a))| < 5\sigma/64, \quad a\in\mathcal F,\ \tau\in\mathrm{T}(Q^r).
\eneq

Since $1_A \in\mathcal P$, by \eqref{prop-2-n-2}, $[\Sigma_2(1_A)]\in\mathbb D$, and so there is  a projection $e_1\in B$ such that $\pi_0(e_1)=\psi_0(\Sigma_2(1_A))$
and $\pi_1(e_1)=\psi_1(\Sigma_2(1_A)).$ {{From the  construction, $$\Psi_i(1_A)=\Phi_i(1_A)=1\oplus\pi_i(e_0)\oplus \pi_i(e_1),\quad i=0, 1.$$}}
It then follows from \eqref{prop-2-3} (applied just for $\tau$ factoring through $\psi_0$---alternatively, for $\tau$ factoring through $\psi_1$) that
\begin{equation}\label{small-2-n-proj}
\tau(e_1) < \sigma/64,\quad \tau\in\mathrm{T}(B).
\end{equation}
Set $E_0'=1\oplus \pi_0(e_0)\oplus \pi_0(e_1),$ $E_1'=1\oplus \pi_1(e_0)\oplus \pi_1(e_1),$
and  $D_0=E_0' {{\mathrm{M}_3}}(Q^r)E_0',$  $D_1=E_1'{{\mathrm{M}_3}}(Q^r)E_1'.$


Pick a sufficiently small $r'\in (0, 1/4)$ that 
\begin{equation}\label{small-pert-hor}
\| \Psi (a)((1+2r')t-r') - \Psi(a)(t) \| <\sigma/64,\quad a\in\mathcal G,\ t\in [\frac{r'}{1+2r'}, \frac{1+r'}{1+2r'}].
\end{equation}


It follows from Lemma \ref{partuniq}
(in view of (\ref{TT-13}), (\ref{TT-13+1}) and (\ref{TT-13+2})) that there exist unitaries $u_0\in D_0$ and
 $u_1\in D_1,$
 and unital ${\mathcal F}$-$\ep$-multiplicative completely positive maps $L_0:A\to {\rm C}([-r', 0], D_0)$ and
$L_1: A\to  {\rm C}([1, 1+r'], D_1)$, such that
\beq
&&\pi_{-r'}\circ L_0=\Phi_0,\,\,\, \pi_0\circ L_0={\rm Ad}\, u_0\circ {{\Psi_0}}, \label{small-var-trace--2} \\
&&\pi_{1+r'}\circ L_1=\Phi_1,\,\,\, \pi_{1}\circ L_1={\rm Ad}\, u_1\circ {{\Psi_1}}, \label{small-var-trace--1} \\
&&|\tau\circ \pi_t\circ L_0(a)-\tau\circ \pi_0\circ L_0(a)|<5\sigma/32,\quad t\in [-r', 0], \label{small-var-trace-0} \\
&&|\tau\circ \pi_t\circ L_1(a)-\tau\circ \pi_1\circ L_1(a)|<5\sigma/32,\quad t\in [1,1+r'], \label{small-var-trace-1}
\eneq
where $a\in {\mathcal F},$ $\tau\in {\rm T}(Q^r),$ and (as before) $\pi_t$ is the point evaluation at $t\in [-r', 1+r'].$ 

Write $E_3=1\oplus e_0\oplus e_1\in \mathrm{M}_3(\mathrm{C}([0,1], Q^r))$ and $B_1=E_3(\mathrm{M}_3(\mathrm{C}([0,1], Q^r)))E_3.$
There exists a  unitary $u\in B_1$ such
that $u(0)=u_0$ and $u(1)=u_1.$
Consider the projection $E_4\in \mathrm{M}_3(\mathrm{C}([-r',1+r'], Q^r))$ defined by    
$E_4|_{[-r,0]}=E_0',$ $E_4|_{[0,1]}=E_3$ and $E_4|_{[1,1+r]}=E_1'.$
Set
$$B_2=E_4(\mathrm{M}_3(\mathrm{C}([-r',1+r'], Q^r)))E_4.$$
Define a unital $\mathcal F$-$\eps$-multiplicative (note that $\mathcal F\subseteq \mathcal G$ and $\delta\leq \eps$) completely positive map $L': A\to B_2$ by
\begin{equation}\label{defn-new-L}
L'(a)(t) = \left\{\begin{array}{ll}
L_0(a)(t), & t\in [-r', 0), \\
{\rm Ad}\, u(t)\circ (\pi_t\circ\Psi\oplus {{{\rm Ad}w(t)\circ \mu_2}})(a), & t\in [0, 1] ,\\
L_1(a)(t) & t \in (1, 1+r'].
\end{array}\right.
\end{equation}

Note that for any $a\in\mathcal G$, and any $\tau\in\mathrm{T}(Q^r)$, by \eqref{defn-new-L}, if $t\in[0, 1]$, then
\begin{eqnarray}\label{pre-realize-trace-1}
&&\hspace{-0.4in}|\tau(\pi_t(L'(a))) - \gamma_1(\pi_t^*(\tau))(a) | \nonumber \\
 & = & |\tau({\rm Ad}\, u(t)\circ ({ {\pi_t\circ}}\Psi\oplus {{{\rm Ad}w(t)\circ \mu_2}})(a)) - \gamma_1(\pi_t^*(\tau))(a)| \nonumber \\
& = & |\tau(\pi_t(\Psi(a))) + \tau(\mu_2(a)) - \gamma_1(\pi_t^*(\tau))(a)| \nonumber \\
& < & |(\pi_t^*(\tau))(\Psi(a))) - \gamma_1(\pi_t^*(\tau))(a)| + \sigma/64 \,\,\,\quad\quad\textrm{(by \eqref{prop-2-n-1} and \eqref{prop-2-3})} \nonumber \\
& < & \min\{\delta_1/16, \sigma/128\} + \sigma/64 \leq 3\sigma/128 \,\,\,\quad\quad\quad\textrm{(by \eqref{TT-9-use-n-1})},
\end{eqnarray}
where $\pi_t^*: \mathrm{T}(Q^r) \to \mathrm{T}(B)$ is the dual of $\pi_t: B\to Q^r$.
Furthermore, if $t \in [-r', 0]$, then for any $a\in\mathcal F$, and any $\tau\in\mathrm{T}(Q^r)$,
\begin{eqnarray}\label{pre-realize-trace-2}
&&\hspace{-0.4in} |\tau(\pi_t(L'(a))) - \gamma_1(\pi_0^*(\tau))(a) | \nonumber \\
&=& |\tau(L_0(a)(t)) - \gamma_1(\pi_0^*(\tau))(a)| \nonumber \\
& < & |\tau(L_0(a)(0)) - \gamma_1(\pi_0^*(\tau))(a)| + 5\sigma/32\,\quad\quad\quad\quad\quad\quad\textrm{(by \eqref{small-var-trace-0})} \nonumber \\
& = & |\tau(\Psi_0(a))- \gamma_1(\pi_0^*(\tau))(a)| + 5\sigma/32 \,\,\quad\quad \quad\quad\quad \quad\quad\textrm{(by \eqref{small-var-trace--2})} \nonumber \\
& = & |\tau((\pi_0\circ \Psi)(a) \oplus {{{\rm Ad}w_0\circ \mu_2}}(a))- \gamma_1(\pi_0^*(\tau))(a)| + 5\sigma/32 \nonumber  \\
& < & |\tau((\pi_0\circ \Psi)(a)- \gamma_1(\pi_0^*(\tau))(a)| + \sigma/64 + 5\sigma/32 \quad\quad\textrm{(by \eqref{prop-2-n-1} and \eqref{prop-2-3})} \nonumber \\
& < & \min\{\delta_1/16, \sigma/128\} + \sigma/64 + 5\sigma/32 <23\sigma/128 \quad\quad\textrm{(by \eqref{TT-9-use-n-1})}.
\end{eqnarray}
Again, if $t \in [1, 1+r']$, then the same argument shows that for any $a\in\mathcal F$, and any $\tau\in\mathrm{T}(Q^r)$,
\begin{equation}\label{pre-realize-trace-3}
|\tau(\pi_t(L'(a))) - \gamma_1(\pi_1^*(\tau))(a) | < 23\sigma/128.
\end{equation}


Let us modify $L'$ to a unital map from $A$ to $B$.
First, let us renormalize $L'$. Consider the isomorphism $\eta: \Q^r\to \Q^r$ defined by
$$
\eta(x_1,x_2,...,x_r)=({1\over{\mathrm{tr}_1(E_3)}}x_1, {1\over{\mathrm{tr}_2(E_3)}}x_2,...,{1\over{\mathrm{tr}_r(E_3)}}x_r),
$$
for all $(x_1,x_2,...,x_r)\in \Q^r,$
where (as before) $\mathrm{tr}_k$ is the tracial state supported on
the $k$th direct summand of $Q^r$.  
Then there is a (unital) isomorphism $\phi: B_2\to {\rm C}([-r', 1+r'], Q^r)$ such that $\phi_{*0}=\eta.$
Let us replace the map $L'$ by the map $\phi\circ L'$, and still denote it by $L'$. Note that it follows from \eqref{pre-realize-trace-1}, \eqref{small-trace-n-2}, and \eqref{small-2-n-proj} that for any $t\in [0, 1]$, any $a\in\mathcal F$, and any $\tau\in\mathrm{T}(Q^r)$,
\begin{eqnarray}\label{pre-realize-trace-2-1}
& & |\tau(\pi_t(L'(a))) - \gamma_1(\pi_t^*(\tau))(a) |  \nonumber \\
& < &  
{{3\sigma/128}}+ \tau(e_0) + \tau(e_1) \nonumber \\
& < &  3\sigma/128 + \min\{\delta_1/16, \sigma/64\}+ \sigma/64 \leq 7\sigma/128.
\end{eqnarray}
The same argument, using \eqref{pre-realize-trace-2} and \eqref{pre-realize-trace-3} instead of \eqref{pre-realize-trace-1}, shows that for any $a\in\mathcal F$,
\begin{eqnarray}
|\tau(\pi_t(L'(a))) - \gamma_1(\pi_0^*(\tau))(a) | & <  & 27\sigma/128,\quad t\in[-r', 0] \label{pre-realize-trace-2-2}, \\
|\tau(\pi_t(L'(a))) - \gamma_1(\pi_1^*(\tau))(a) | & < & 27\sigma/128, \quad t\in[1, 1+r'] \label{pre-realize-trace-2-3}.
\end{eqnarray}

Now, put
\begin{equation}\label{defn-new-new-L}
L''(a)(t) = L'(a)((1+2r')t-r'),\quad t\in[0, 1].
\end{equation}
This perturbation will not change the trace very much,
as for any $a\in\mathcal F$ and any $\tau\in \mathrm{T}(Q^r)$,  if $t\in[0, {r'}/(1+2r')]$, then
\begin{eqnarray*}
&  &\hspace{-0.4in} |\tau(L''(a)(t)) - \tau(L'(a)(t))| \\
&=& |\tau(L'(a)((1+2r')t-r')) - \tau(L'(a)(t))|  \,\,\,\quad\quad\quad\quad\quad\textrm{(by \eqref{defn-new-new-L})} \\
& = & |\tau(L_0(a)((1+2r')t-r')) - \tau((\Psi(a)(t)) +{{{\rm Ad}w(t)\circ \mu_2}}(a))|  \\
& < & |\tau(L_0(a)((1+2r')t-r')) - \tau(\Psi(a)(t)) | + \sigma/64 \quad\quad\textrm{(by \eqref{prop-2-n-1}) and \eqref{prop-2-3}} \\
& < & |\tau(L_0(a)(0)) - \tau(\Psi(a)(t)) | + 5\sigma/32 + \sigma/64 \,\,\,\,\,\quad\quad\quad\textrm{(by \eqref{small-var-trace-0})} \\
& = &  |{{\tau(\Psi_0(a))}}) - \tau(\Psi(a)(t)) | + 11\sigma/64  \,\,\,\quad\quad \quad\quad\quad\quad\textrm{(by \eqref{small-var-trace--2})} \\
& < & \sigma/64 + 11\sigma/64=3\sigma/16   \,\, \,\,\,\quad\quad\quad\quad\quad\quad\quad\quad\quad\quad\quad\textrm{(by \eqref{prop-2-n-1}) and \eqref{prop-2-3}}.
\end{eqnarray*}
Furthermore, the same argument, now using \eqref{small-var-trace-1} and \eqref{small-var-trace--1}, shows that for any $a\in\mathcal F$, $\tau\in\mathrm{T}(Q^r)$, and $t\in[(1+r')/(1+2r'), 1]$,
\begin{equation*}
|\tau({{L'(a)}}((1+2r')t-r')) - (\tau(\Psi(a)(t)) + {{{\rm Ad}w_1\circ\mu_2}}(a))| < 3\sigma/16,
\end{equation*}
and, if $t\in [{r'}/(1+2r'), (1+r')/(1+2r')]$, then
\begin{eqnarray*}
|\tau(L'(a)((1+2r')t-r')) - \tau(L'(a)(t))| & = & | \tau(\Psi (a)((1+2r')t-r') - \Psi(a)(t))| \\
& < & \sigma/64  \quad\quad\quad\quad\textrm{(by \eqref{small-pert-hor})}.
\end{eqnarray*}
Thus,
\begin{equation}\label{small-pert-n-2}
|\tau(L''(a)(t)) - \tau(L'(a)(t))| < 3\sigma/16,\quad a\in\mathcal F,\ \tau\in\mathrm{T}(Q^r),\ t\in[0, 1].
\end{equation}
Hence,
\begin{eqnarray}\label{2nd-pert-trace}
&&|\tau(\pi_t(L''(a))) - \gamma_1(\pi_t^*(\tau))(a)| \nonumber\\
&\leq & |\tau(L''(a)(t)) - \tau(L'(a)(t))| + |\tau(L'(a)(t))- \gamma_1(\pi_t^*(\tau))(a)| \nonumber \\
& < & 3\sigma/16 + 27\sigma/128 =51\sigma/128 \quad\quad \textrm{(by \eqref{small-pert-n-2}, \eqref{pre-realize-trace-2-1}, \eqref{pre-realize-trace-2-2}, and \eqref{pre-realize-trace-2-3})}.
\end{eqnarray}

Note  that  $L''$ is a unital map from $A$ to $B$. It is also $\mathcal F$-$\eps$-multiplicative, since $L'$ is.  

Consider the order isomorphism $\eta': \Q^l\to \Q^l$
defined by
$$
\eta'(y_1,y_2,...,y_l)=(a_1y_1, a_2y_2,...,a_ly_l),\quad (y_1,y_2,...,y_l)\in \Q^l,
$$
where
\begin{equation}\label{scaling-const}
a_j={1\over{{\rm tr}_j(1\oplus \Sigma_1(1_A) \oplus \Sigma_2(1_A))}},\quad  j=1,2,...,l,
\end{equation}
and (as before) $\mathrm{tr_j}$ is the tracial state supported on the $j$th direct summand of $Q^l$.
There exists a unital \hm\, ${\tilde \phi}: (1\oplus \Sigma_1(1_A) \oplus \Sigma_2(1_A))\mathrm{M}_3(Q^l)(1\oplus \Sigma_1(1_A) \oplus \Sigma_2(1_A))
\to Q^l$ such that
\begin{equation}\nonumber
{\tilde \phi}_{*0}=\eta'.
\end{equation}
Therefore, by the constructions of $L''$, $L'$, $L_0$, and $L_1$ (\eqref{defn-new-new-L}, \eqref{defn-new-L}, \eqref{small-var-trace--2}, and \eqref{small-var-trace--1}), we may assume
that
\begin{equation}\label{mt-cond}\psi_0\circ {\tilde \phi}\circ (\Phi'\oplus \Sigma_2)=\pi_0\circ L''\quad
\textrm{and}\quad \psi_1\circ {\tilde \phi}\circ (\Phi'\oplus \Sigma_2)=\pi_1\circ L'',
\end{equation}
replacing $L''$ by {{${\rm Ad}\, v\circ L''$}}  for a suitable unitary {{$v$}} if necessary.

 Define $L: A\to C_1$ by
 $L(a)=(L''(a), {\tilde \phi}( \Phi'(a)\oplus \Sigma_2(a)))$, an element of $C_1$ by \eqref{mt-cond}. Since $L''$ and $\tilde{\phi}\circ(\Phi'\oplus\Sigma_2)$ are unital and ${\mathcal F}$-$\ep$-multiplicative (since $\Phi'$ and $\Sigma_2$ are $\mathcal G$-$\delta$-multiplicative, $\mathcal F\subseteq \mathcal G$, and $\delta\leq\eps$), so also is $L$.


Moreover, for any $a\in\mathcal F$, any $\tau\in\mathrm{T}(Q^r)$, and any $t\in (0, 1)$, it follows from \eqref{2nd-pert-trace} that
\begin{equation}\label{approx-n-1}
|\tau(\pi_t(L(a))) - \gamma_1(\pi_t^*(\tau))(a)|  <  51\sigma/128.
\end{equation}
If $\tau\in\mathrm{T}(Q^l)$, then, for any $a\in\mathcal F$,
\begin{eqnarray*}
&&\hspace{-0.4in} |\tau(\pi_{\mathrm{e}}(L(a))) - \gamma^*(\pi_{\mathrm{e}}^*(\tau))(a)| \\
& = & |\tau( {\tilde \phi}( \Phi'(a)\oplus \Sigma_2(a))) - \gamma^*(\pi_{\mathrm{e}}^*(\tau))(a)| \\
& < & |\tau(\Phi'(a)\oplus \Sigma_2(a)) - \gamma^*(\pi_{\mathrm{e}}^*(\tau))(a)|+\sigma/32 \quad\quad\textrm{(by \eqref{scaling-const}, \eqref{small-trace-n-2} and \eqref{prop-2-3})}\\
& < & |\tau(\Phi'(a)) - \gamma^*(\pi_{\mathrm{e}}^*(\tau))(a)| + 3 \sigma/64  \,\quad \quad\quad\quad\quad\textrm{(by \eqref{prop-2-3})} \\
& < & |\tau(\Phi(a)) - \gamma^*(\pi_{\mathrm{e}}^*(\tau))(a)| +  \sigma/8  \,\, \quad\quad\quad\quad\quad\quad\textrm{(by \eqref{pert-1-1})} \\
& < & \sigma/32 + \sigma/8<51\sigma/128    \quad\quad\quad\quad \quad\quad\quad\quad\quad\quad\textrm{(by \eqref{trace-infty})}.
\end{eqnarray*}
Since each  extreme trace of $C_1$ factors through either the evaluation map $\pi_t$ or the canonical quotient map $\pi_{\mathrm{e}}$,  by \eqref{approx-n-1},
\begin{equation}\label{approx-n-1-1}
|\tau(L(a)) - \gamma^*(\tau)(a)|  <  51\sigma/128,\quad \tau\in \mathrm{T}(C_1),\ a\in\mathcal F.
\end{equation}
Therefore,
for any $a\in {\mathcal F}$ and $\tau\in\mathrm{T}(C)$, we have
\begin{eqnarray*}
 &&\hspace{-0.4in}|\tau(\imath_{1, \infty}(L(a)))-\Gamma_\aff(\hat{a})(\tau)| \\
 &< & |\tau(\imath_{1, \infty}(L(a)))-\gamma^*(\imath_{1, \infty}(\tau))(a)|+|\gamma_*(\imath_{1, \infty}(\tau))(a)-\Gamma_\aff(\hat{a})(\tau)| \\
 &=&|\tau(\imath_{1, \infty}(L(a)))-\gamma^*(\imath_{1, \infty}(\tau))(a)|+|(\imath_{1, \infty})_{\aff}(\hat{a})(\tau)-\Gamma_\aff(\hat{a})(\tau)| \\
 & < & 51\sigma/128 + 77\sigma/128 = \sigma  \quad\quad\quad\quad \textrm{(by \eqref{approx-n-1-1} and \eqref{TT-1})}.
 \end{eqnarray*}

Since $\mathcal F$, $\eps$, and $\sigma$ are arbitrary, in this way we obtain a sequence of unital completely positive maps $H_n: A\to C$ such that
\begin{equation*}
\lim_{n\to\infty}\|H_n(ab)-H_n(a)H_n(b)\|=0,\quad  a,b\in A,
\end{equation*}
and
\begin{equation*}
\lim_{n\to\infty}\sup\{|\tau\circ H_n(a)-\Gamma_{\aff}(a)(\tau)|: \tau\in {\rm T}(C)\}
=0,\quad a\in A.
\end{equation*}
On using again that the given C*-algebra $A$ has finite nuclear dimension, so that also $A\otimes Q$ does, it follows by Lemma 3.4 of \cite{Lncross}---which uses results obtained in \cite{Rb} and \cite{W2}---that $\mathrm{gTR}(A\otimes Q)\leq 1$. This together with $\mathcal Z$-stability of $A$ (which we established at the very beginning of this proof) says that the given algebra $A$ belongs to the class $\mathcal N_1$.
%
\end{proof}

{{Theorem \ref{mainthm-dr} follows from the following corollary.}}

\begin{cor}\label{MC}
Let $A$ be a unital simple separable {{C*-algebra with finite decomposition rank,}} satisfying the UCT.
Then ${\mathrm{gTR}}(A\otimes Q)\le 1.$ In particular, $A$ is classifiable.

\end{cor}

\begin{proof}
Since $A$  has finite decomposition rank,  $A$ is nuclear (see \ref{DefDr} {{above}}) and
quasidiagonal (5.3 of \cite{KW}).  It follows from 2.4 of \cite{V} that ${\rm T}(A)\not=\O.$
By {{Proposition}}
 8.5 of \cite{BBSTWW},
 ${\rm T}(A)={\rm T}_{{\rm qd}}(A).$
Now the corollary follows from Theorem \ref{TT} together with Theorem \ref{Tiso}.

\end{proof}

\begin{rem}\label{Rrtr=1}
We would like to state the following special case of  Corollary \ref{MC}. Let
$A$ be as in \ref{MC}.  Suppose that
\beq\label{Leq}
&&\hspace{-0.65in}(K_0(A\otimes Q), K_0(A\otimes Q)_+, [1_{A\otimes Q}]_0, {\mathrm T}(A\otimes Q), r_{A\otimes Q})
\cong (K_0(C), K_0(C)_+, [1_C]_0, {\mathrm T}(C), r_C)
\eneq
for some  unital simple ${\mathrm{A}}\T$-algebra $C.$ {T}hen ${\mathrm{TR}}(A\otimes Q)\le 1.$ {If this holds}
for some   unital AF-algebra $C,$ then ${\mathrm{TR}}(A\otimes Q)=0.$

To see this we  note that, in the beginning of the proof of  Theorem \ref{TT},
we assume that $A=A\otimes Q.$ If $C$ in \eqref{Leq} can be chosen to be a unital simple ${\mathrm A}\T$-algebra,
then the end of the proof shows that $A\otimes Q$ has tracial rank at most one. {In the same way one sees that} if $C$ in \eqref{Leq} can be chosen to be a unital simple AF-algebra, then $A\otimes Q$ has tracial rank zero.

The preceding (abstract) classification result {(Corollary \ref{MC})} depends on (by reducing to) the recent (semi-abstract) classification result of
\cite{GLNI} and \cite{GLNII}. In fact, there is a more restricted, but still very interesting, setting
in which a correspondingly restricted abstract result can be established by reducing to a much earlier
result.

Let $A$ be a unital simple separable {C*-algebra}, satisfying the hypotheses
of the preceding corollary (or theorem).
Suppose in addition that  {$\mathrm{S}_{[1_A]}(\mathrm{K}_0(A)),$} the state space of {$\mathrm{K}_0(A),$}   is
a Choquet simplex, and that the map $r_A: {\mathrm{T}(A)\to \mathrm{S}_{[1_A]}(\mathrm{K}_0(A))}$
takes extreme points to extreme points.
Without using \cite{GLNI} or \cite{GLNII},
the proof of
the {present} Theorem \ref{TT},
above, shows
that $A$ is classifiable.

Indeed, by \cite{LN2}, there is a unital simple separable {C*-algebra}, $B$,  satisfying
the UCT, such that
$B\otimes Q$ has tracial rank at most one (in the sense of \cite{LinTAI}), and  such that
${\rm Ell}(A)={\rm Ell}(B).$  Since ${\mathrm{K}}_i(B\otimes Q)$ is torsion free,
by the classification of {C*-algebras} of tracial rank at most one (see \cite{LinTAI} and \cite{EGL}),
$C=B\otimes Q$ is an inductive limit of circle algebras ({i.e., is $\mathrm{A}\mathbb T$}).
By the first paragraph of this remark,
$A\otimes Q$ has tracial rank at most one.  {Hence by} {{Corollary 11.9  of \cite{Lininv}
(see also \cite{LS}),}}
$A$ is classifiable.

In particular, the Jiang-Su algebra ${\cal Z}$ is the only unital  separable simple amenable
C*-algebra  {{in the UCT class that}} has the same Elliott invariant as that of $\C.$
The proof just {{given}} of this statement does not rely on \cite{GLNI} or \cite{GLNII}.
\end{rem}

\begin{rem}\label{Refiniterank}
It was shown in \cite{ENST} that any unital simple separable Jiang-Su stable approximately subhomogeneous C*-algebra has decomposition rank at most two.   Therefore, it follows from Corollary \ref{MC} that such a C*-algebra is classifiable.
This in particular recovers the classification theorem of \cite{EGLN}.
Moreover, by \cite{GLNII} together with the result in \cite{ENST} mentioned above, every unital simple {C*-algebra},
belonging to the class ${\cal N}_1$ has finite decomposition rank.
\end{rem}

\begin{rem}
The special case of Corollary \ref{MC} for C*-algebras
for which K$_0$ separates traces, e.g. the case of unique trace, is known.
(See Corollary 5.2 of \cite{Wdecomp} and Theorem 5.4 of \cite{LN}.)
\end{rem}

Theorem \ref{TT} and Corollary \ref{MC} can also be combined and stated as follows:

\begin{thm}\label{T2}
Let $A$ be a unital  simple separable amenable (non-zero) C*-algebra  which satisfies the UCT.
Then the following  properties
are equivalent:

{\rm (1)} ${\mathrm{gTR}}(A\otimes Q)\le 1;$

{\rm (2)} $A\otimes Q$ has finite nuclear dimension and ${\rm T}(A\otimes Q)={\rm T}_{{\rm qd}}(A\otimes Q)\not=\O;$

{\rm (3)} the decomposition rank of
$A\otimes Q$
is finite.
\end{thm}

Given
the  fact that
every tracial state of  a unital simple separable C*-algebra with finite decomposition rank
is quasidiagonal ({{Proposition}} 8.5 of \cite{BBSTWW}),
it is reasonable to expect that  every tracial state of  a finite unital   simple separable  C*-algebra with finite
nuclear dimension is also quasidiagonal.
 Indeed,
 shortly after the present paper was first announced (and posted on arXiv), Tikuisis, White, and Winter proved that, in fact,
every tracial state on a unital simple separable amenable C*-algebra which satisfies the UCT is quasidiagonal (Theorem A of \cite{TWW}).
Therefore, we have  the  following statement:

\begin{thm}\label{NT}
Let $A$ be a finite unital  simple separable   C*-algebra  with finite nuclear dimension which satisfies the UCT.  Then ${\mathrm{gTR}}(A\otimes Q)\le 1$. In particular, $A$ is classifiable (and is approximately subhomogeneous (ASH)---see
Theorems \ref{Tiso} and \ref{Trange}).
\end{thm}

\begin{rem}\label{Lrem}

It was established by Kirchberg and Phillips  (\cite{KP} and \cite{P}) that purely infinite unital
 simple separable  amenable C*-algebras
which satisfy the UCT are classifiable.  It has been shown that these C*-algebras have finite nuclear dimension (see\,\cite{MS}). It is also known that
every unital simple separable C*-algebra with finite nuclear dimension is either finite or purely infinite
(see \cite{GJS} and \cite{W}).
Therefore, Theorem \ref{NT} can now  be combined with \cite{KP} and \cite{P} to obtain the following overall statement.


\end{rem}

\begin{cor}\label{KP}
The class of all unital simple separable (non-elementary) C*-algebras with finite nuclear dimension
which satisfy the UCT is classifiable by the Elliott invariant.
\end{cor}

{\bf Added {November 2}, 2021}

{This} paper was originally posted on arXiv in late 2015. Since then, there have been some new developments.
Notably, in \cite{CETWW} (also \cite{MS2} and \cite{SWW}), it was shown
that every Jiang-Su  stable (${\cal Z}$-stable) unital simple separable amenable C*-algebra  has finite nuclear dimension.
This yields  many examples  which come {under}  the aegis of Corollary \ref{KP}.  For instance,  since
by \cite{JS}, the Jiang-Su algebra, ${\cal Z},$ is itself ${\cal Z}$-stable,  the tensor product of any C*-algebra $A$ with ${\cal Z}$
is ${\cal Z}$-stable, and so by \cite{CETWW}, if $A$ is a unital simple  separable amenable C*-algebra, and satisfies the UCT,
then $A\otimes {\cal Z}$ is covered by \ref{KP}.  Sometimes, it can be established  without tensoring by ${\cal Z}$
that a {{C*-algebra}} is Jiang-Su stable.  For example,  this was achieved in \cite{EN-m} for a simple C*-algebra arising
from a  minimal homeomorphism  of an infinite compact metrizable space  of mean dimension zero
 (which
includes the cases that the space is finite dimensional, or has a unique invariant measure).

 \vskip 4mm
 Department of Mathematics, University of Toronto, Toronto, Ontario, Canada\ \  M5S 2E4 \par
 {\small {\it E-mail address:} \tt{elliott@math.toronto.edu}}

 \vskip 4mm
 Department of Mathematics, University of Puerto Rico, San Juan, PR 00936, USA \par
 {\small {\it E-mail address:} \tt{ghgong@gmail.com}}

 \vskip 4mm
 Department of Mathematics, University of Oregon, Eugene, OR 97403, USA \par
 {\small {\it E-mail address:} \tt{hlin@uoregon.edu}}

 \vskip 4mm
 Department of Mathematics, University of Wyoming, Laramie, WY 82071, USA \par
 {\small {\it E-mail address:} \tt{zniu@uwyo.edu}}

\end{document}